\documentclass[11pt]{amsart}
\usepackage{amsmath, amsfonts, amssymb,amsthm}
\usepackage{euscript}
\usepackage[T1]{fontenc}
\usepackage{graphicx}
\usepackage{color}

\usepackage{hyperref}

\newtheorem{thm}{Theorem}[section]
\newtheorem{lem}[thm]{Lemma}

\newtheorem{prop}[thm]{Proposition}
\newtheorem{cor}[thm]{Corollary}

\theoremstyle{definition}
\newtheorem{defn}[thm]{Definition}

\newtheorem{ex}[thm]{Example}
\newtheorem{rem}[thm]{Remark}

\DeclareMathOperator{\R}{\mathbb R}
\DeclareMathOperator{\C}{\mathbb C}
\DeclareMathOperator{\N}{\mathbb N}

\DeclareMathOperator{\Z}{\mathcal Z}
\DeclareMathOperator{\V}{\mathcal V}
\DeclareMathOperator{\I}{\mathcal I}
\DeclareMathOperator{\D}{\mathcal D}
\DeclareMathOperator{\SR}{\mathcal K^0}
\DeclareMathOperator{\SRR}{\mathcal R^0}

\DeclareMathOperator{\SO}{\mathcal O}

\DeclareMathOperator{\K}{\mathcal K}

\DeclareMathOperator{\supp}{supp}

\DeclareMathOperator{\Cent}{Cent}

\DeclareMathOperator{\RCent}{C-Spec}
\DeclareMathOperator{\RSp}{R-Spec}

\DeclareMathOperator{\Sp}{Spec}
\DeclareMathOperator{\ReSp}{R-Spec}

\DeclareMathOperator{\p}{\mathfrak{p}}
\DeclareMathOperator{\q}{{\mathfrak q}}
\DeclareMathOperator{\ir}{{\mathfrak r}}

\DeclareMathOperator{\Max}{\rm Max}
\DeclareMathOperator{\ReMax}{\rm R-Max}
\DeclareMathOperator{\CentMax}{\rm C-Max}

\DeclareMathOperator{\m}{\mathfrak m}

\def\JRadRe {{\rm{Rad}^R}}
\def\JRadCe {{\rm{Rad}^C}}
\def\JRad {{\rm{Rad}}}
\def\RadRe {\sqrt[R]}
\def\Rad {\sqrt}
\def\RadCe {\sqrt[C]}

\begin{document}

\title[\tiny{Central algebraic geometry and seminormality}]{Central
  algebraic geometry and seminormality}
\author{Jean-Philippe Monnier}
\thanks{The author benefit from the support of the Centre Henri Lebesgue ANR-11-LABX-0020-01 and from the project ENUMGEOM ANR-18-CE40-0009.}

\address{Jean-Philippe Monnier
Univ Angers, CNRS, LAREMA, SFR MATHSTIC, F-49000 Angers, France}
\email{jean-philippe.monnier@univ-angers.fr}

\date\today

\begin{abstract} We develop the theory of central ideals on commutative rings. We introduce and study the central seminormalization of a ring in another one. This seminormalization is related to the theory of regulous functions on real algebraic varieties. We provide a construction of the central seminormalization by a decomposition theorem in elementary central gluings. The existence of a central seminormalization is established in the affine case and for real schemes.
\end{abstract}

\maketitle

\section{Introduction}

The present paper is devoted to the study of the seminormalization in the real setting. The operation of seminormalization was formally introduced around fifty years ago first
in the case of analytic spaces by Andreotti and Norguet \cite{AN} and later in the abstract scheme setting by Andreotti and Bombieri \cite{AB}. The notion arose from a classification problem. For algebraic varieties, the seminormalization  of $X$ in $Y$ is basically the biggest intermediate variety which is bijective with $X$. Recently, the concept of seminormalization appears in the study of singularities of algebraic varieties, in particular in the minimal model program of Koll\'ar and Kov\'acs (see \cite{KoKo} and \cite{Ko}).

Around 1970 Traverso \cite{T} introduced the closely related notion of the seminormalization $A_B^*$ of a commutative ring $A$ in an integral extension $B$. The idea is to glue together the prime ideals of $B$ lying over the same prime ideal of $A$. The seminormalization $A_B^*$ has the property that it is the biggest extension of $A$ in a subring $C$ of $B$ which is subintegral i.e such that the map $\Sp C\to\Sp A$ is bijective and equiresidual (it gives isomorphisms between the residue fields). For geometric rings all these notions of seminormalizations are equivalent and are strongly related with the Grothendieck notion of universal homeomorphism \cite[I 3.8]{Gr1}. We refer to Vitulli \cite{V} for a survey on seminormality for commutative rings and algebraic varieties. See also \cite{GT}, \cite{LV}, \cite{Sw} and \cite{Vcor} for more detailed informations on seminormalization.

For an integral extension $B$ of a commutative ring $A$, using the classical notion of real ideal \cite{BCR}, we may try to copy Traverso's construction by gluing together all the real prime ideals of $B$ lying over the same real prime ideal of $A$. Unfortunately it doesn't give an acceptable notion of real seminormalization since real prime ideals do not satisfy a lying-over property for integral extensions. Normalization in the real setting is deeply studied in \cite{FMQ-futur}, the aim of the paper is to develop the theory of central seminormalization introduced in \cite{FMQ-futur2}.

The paper is organized as follows. In the second section we recall some classical results on real algebra and more precisely about the theory of real ideals as it is developed in \cite{BCR}. In the third section we introduce the notion of 
central ideal: If $I$ is an ideal of an integral domain $A$ with fraction field $\K(A)$, we say that $I$ is a central ideal if $\forall a\in A$, $\forall b\in\sum\K(A)^2$ ($\sum\K(A)^2$ is the set of sum of squares of elements in $\K(A)$) we have
$$a+b\in I\Rightarrow a\in I.$$
We develop the theory of central ideals
similarly to the theory of real ideals done in \cite{BCR} proving in particular that an ideal is central if and only if it is equal to its central radical (the intersection of the central prime ideals containing it). We prove that the notion of central ideal developed here is compatible for geometric rings (coordinate rings of affine variety over $\R$) with the Central Nullstellensatz 
\cite[Cor. 7.6.6]{BCR} and also coincides for prime ideals with that of \cite{FMQ-futur2}. For a domain $A$, the central spectrum of $A$ (the set of central prime ideals of $A$) is denoted by $\RCent A$. For an extension $A\to B$ of domains we show that we have a well defined associated map $\RCent B\to \RCent A$. In the fourth section we show 
that central ideals (that are real ideals) behave much better than real ideals when we consider integral extensions of rings. This is the principal reason we prefer working with central ideals in this paper. Especially, we have the following lying-over property: Let $A\to B$ be an integral and birational extension of domains (birational means $\K(A)\simeq \K(B)$),
then $\RCent B\to \RCent A$ is surjective.

Regarding the classical case, we say that an extension $A\to B$ of domains is centrally subintegral if it is an integral extension such that the associated map $\RCent B\to\RCent A$ is bijective and equiresidual. Surprisingly, centrally subintegral extensions of geometric rings are strongly linked with the recent theory of rational continuous and regulous functions on real algebraic varieties introduced by Fichou, Huisman, Mangolte and the author \cite{FHMM} and by Koll\'ar and Nowak \cite{KN}. Let $X$ be an irreducible affine algebraic variety over $\R$ with coordinate ring $\R[X]$. The central locus $\Cent(X)$ is the subset of the set of real closed points $X(\R)$ such that the associated ideal is central. By the Central Nullstellensatz $\Cent X$ coincides with the Euclidean closure of the set of smooth real closed points. Following \cite{FMQ-futur2}, we denote by $\SR(\Cent X)$, called
the ring of rational continuous functions on $\Cent X$, the ring
of continuous functions on $\Cent X$ that are rational on $X$. We denote by $\SRR (\Cent X)$, called the ring
of regulous functions on $\Cent X$, the subring of $\SR(\Cent X)$
given by rational continuous functions that satisfies the additional
property that they are still rational by restriction to any subvariety intersecting $\Cent X$ in maximal dimension. The link between centrally subintegral extensions and regulous functions is given by the following result: Given a finite morphism $\pi:Y\to X$ between two irreducible affine algebraic varieties over $\R$ then $\pi^*:\R[X]\to\R[Y]$ is centrally subintegral iff the map $\pi_{\mid \Cent Y}\Cent Y\to\Cent X$ is biregulous iff $X$ and $Y$ have the same regulous functions i.e the map $\SRR(\Cent X)\to\SRR(\Cent Y)$, $f\mapsto f\circ\pi_{\mid \Cent Y}$ is an isomorphism. The rational continuous and regulous functions are now extensively studied in real geometry, we refer for example to \cite{KuKu1, KuKu2, FMQ, Mo} for further readings related to the subject.

Similarly to the standard case then we prove in the fifth section that given an extension $A\to B$ of domains there is a biggest extension of $A$ in a subring of $B$ which is centrally subintegral. The target of this biggest extension is denoted by $A^{s_c,*}_B$ and is called the central seminormalization of $A$ in $B$.
This result is a deep generalization of \cite[Prop. 2.23]{FMQ-futur2}. To get the existence of such seminormalization we have introduced and studied several concepts: the central gluing of an integral extension, the birational and birational-integral closure of a ring in another one.

In the sixth section we obtain the principal result of the paper. We have proved the existence of a central seminormalization of a ring in another one but if we take an explicit geometric example i.e a finite extension of coordinate rings of two irreducible affine algebraic varieties over $\R$, due to the fact that when we do the central gluing then we glue together infinitely many ideals, it is in general not easy to compute the central seminormalization. In the main result of the paper, we prove that, under reasonable hypotheses
on the extension $A\to B$, we can obtain the central seminormalization $A^{s_c,*}_B$ from $B$ by  a birational gluing  followed by a finite number of successive elementary central gluings almost like Traverso's decomposition
theorem for classical seminormal extensions \cite{T}. We use this construction to compute the central seminormalization in several examples. This decomposition result allows to prove in section 7 that the processes of central seminormalization and localization commute
together. The proof of the decomposition theorem make strong use of the results on central ideals developed in Section \ref{sectCalg}.

The last section of the paper is devoted to the existence, given a finite type morphism $\pi:Y\to X$ of irreducible affine algebraic varieties over $\R$ or integral schemes of finite type over $\R$, of a central seminormalization of $X$ in $Y$ denoted by $X^{s_c,*}_Y$. It can be seen as a real or central version of Andreotti and Bombieri's construction of the classical seminormalization of a scheme in another one \cite{AB}. We show that the ring $\R[X]^{s_c,*}_{\R[Y]}$ is a finitely generated algebra over $\R$ in the affine case and the $\SO_X$-algebra  $(\SO_X)^{s_c,*}_{\pi_*\SO_Y}$ is a coherent sheaf when we work with schemes. In the affine case, we prove that the coordinate ring of $X^{s_c,*}_Y$
is the integral closure of the coordinate ring of $X$ in a certain ring of regulous functions generalizing one of the main results in \cite{FMQ-futur2}.

\vskip 10mm

{\bf Acknowledgment} :  The author is deeply grateful to G. Fichou and R. Quarez for useful discussions.

\section{Real algebra}

Let $A$ be ring. We assume in the paper that all the rings are commutative and contain ${\mathbb Q}$. 

Recall that an ideal $I$ of $A$ is called
real if, for every sequence $a_1,\ldots,a_k$ of elements of $A$, then
$a_1^2+\cdots+a_k^2\in I$ implies $a_i\in I$ for $i=1,\ldots,k$. 
We denote by $\Sp A$ (resp. $\RSp A$)
the (resp. real) Zariski spectrum of $A$, i.e the set of all
(resp. real) prime ideals of $A$. The set of maximal (resp. and real)
ideals is denoted by $\Max A$ (resp. $\ReMax A$). We endow $\Sp A$
with the Zariski topology whose closed subsets are given by the sets
$\V(I)=\{\p\in\Sp A|\, I\subset\p\}$ where $I$ is an ideal of $A$. If $f\in A$ we denote simply $\V((f))$ by $\V(f)$. The
subsets $\RSp A$, $\Max A$ and $\ReMax A$ of $\Sp A$ are endowed with
the induced Zariski topology. The radical of $I$, denoted by $\Rad I$, is defined as follows:
  $$\Rad I=\{a\in A|\,\, \exists m\in\N \,\,a^m\in I\},$$ it is also the intersection of
  the prime ideals of $A$ that contain $I$. If $B$ is a ring, we denote in the sequel by $\sum B^2$ the set of (finite) sums of squares of elements of $B$. The
real radical of $I$, denoted by $\RadRe I$, is defined as follows:
  $$\RadRe I=\{a\in A|\,\, \exists m\in\N \,\, \exists b\in \sum A^2\,\, {\rm such\,\,that}\,\,a^{2m}+b\in I
  \}.$$
We have:
\begin{prop} \cite[Prop. 4.1.7]{BCR}
\label{propradreal}
\begin{enumerate}
\item[1)] $\RadRe I$ is the smallest real ideal of $A$ containing
  $I$.
\item[2)] $$\RadRe I=\bigcap_{\p\in\ReSp A,\,I\subset\p}\p.$$
\end{enumerate}
\end{prop}
It follows that $I$ is a real ideal if and only if $I=\RadRe I$ and
that a real ideal is radical.

An order $\alpha$ in $A$ is given by a real prime ideal $\p$ of $A$ (called
the support of $\alpha$ and denoted by $\supp(\alpha)$) and an ordering on the residue 
field $k(\p)$ at $\p$. An order can equivalently be given by a morphism $\phi$
from $A$ to a real closed field (the kernel is then the support). The set of orders of $A$ is called
the real spectrum of $A$ and we denote it by $\Sp_r A$. One endows
$\Sp_r A$ with a natural topology whose open subsets are generated by
the sets $\{\alpha\in\Sp_r A|\,\alpha(a)>0\}$. Let $\phi:A\rightarrow
B$ be a ring morphism. It canonically induces continuous maps
$\Sp B\rightarrow \Sp A$, $\RSp B\to\RSp A$ and 
$\Sp_r B\rightarrow \Sp_r A$.

Assume $X=\Sp \R[X]$ is an affine algebraic
variety over $\R$ with coordinate ring $\R[X]$ (see \cite{Man} for a description of the different notions of real algebraic varieties), we
denote by $X(\R)$ the set of real closed points of $X$. We
recall classical notations. If
$f\in\R[X]$ then $\Z(f)=\V(f)\cap X(\R)=\{x\in X(\R)|\,\,f(x)=0\}$ is the real zero
set of $f$.
If $A$
is a subset of $\R[X]$ then $\Z(A)=\bigcap_{f\in A}\Z(f)$ is the real zero
set of $A$. If $W\subset X(\R)$ then
$\I(W)=\{f\in\R[X]|\,\,W\subset\Z(f)\}$ is an ideal called the ideal
of functions vanishing on $W$. We recall the real Nullstellensatz
\cite[Thm. 4.1.4]{BCR}:
\begin{thm} (Real Nullstellensatz)\\
  \label{RealNull}
  Let $X$ be an affine algebraic variety over
  $\R$. Then:
  $$I\subset\R[X]\,\,{\rm is\,\, a\,\, real\,\, ideal}\,\, \Leftrightarrow\,\, I=\I(\Z(I)).$$
\end{thm}

\begin{cor}
 \label{corRealNull}
  Let $X$ be an affine algebraic variety over
  $\R$. The map $\ReMax \R[X]\to X(\R)$, $\m\mapsto \Z(\m)$ is bijective and for any ideal $I\subset\R[X]$ we have 
  $$\Z(I)=\V(I)\cap\ReMax \R[X].$$
\end{cor}

In the sequel we will identify $\ReMax \R[X]$ and $X(\R)$ for an affine algebraic variety $X$ over $\R$. We can endow $X(\R)$ with the induced Zariski topology, the closed subsets are of the form $\Z(I)$ for $I$ an ideal of $A$.

\section{Central algebra}
\label{sectCalg}

The goal of this section is to develop the theory of central ideals
similarly to the theory of real ideals done in \cite{BCR} or in the
previous section. We also prove that the notion of central ideal developed here coincides for prime ideals with that of \cite{FMQ-futur2}.
This section will serve as a basis for developing the theory of central seminormalization and especially to prove a central version of Traverso's decomposition theorem.

In this section $A$ is a domain containing ${\mathbb Q}$. We denote by $\K(A)$ its fraction field. 
\begin{prop}
\label{defrealdomain}
The following properties are equivalent:
\begin{enumerate}
\item[1)] $-1\not\in\sum\K(A)^2$.
\item[2)] $\Sp_r\K(A)\not=\emptyset$.
\item[3)] $(0)$ is a real ideal of $\K(A)$.
\item[4)] $(0)$ is a real ideal of $A$.
\end{enumerate}
\end{prop}

\begin{proof}
See the first chapter of \cite{BCR} to get the equivalence between the first three properties. Since the contraction of a real ideal is a real ideal then 3) implies 4).
Assume $(0)$ is a real ideal of $A$ and $-1\in\sum\K(A)^2$. We have $-1=\sum_{i=1}^n(\frac{a_i}{b_i})^2$ with the $a_i$ and $b_i$ in $A\setminus\{0\}$ and consequently 
$$(\prod_{i=1}^n b_i)^2+\sum_{i=1}^n (a_i(\prod_{j=1,j\not= i}^n b_j))^2=0$$ and since $(0)$ is a real ideal of $A$ then it follows that $a_i=0$ for $i=1,\ldots,n$, impossible.
\end{proof}

In the sequel, we
say that $A$ is a real domain if the equivalent properties of Proposition \ref{defrealdomain} are satisfied. In case $A$ is the coordinate ring $\R[X]$ of an irreducible affine algebraic variety $X$ over $\R$ then 
we simply denote $\K(\R[X])$ by $\K(X)$ and it corresponds to the field of classes of rational functions on $X$ and we call 
it the field of rational functions on $X$ or the function field of $X$.

We modify a bit the definition of a real ideal.
\begin{defn}
  \label{defidealcentral}
Let  $I$ be an ideal of $A$. We say that $I$
is central if for every $a\in A$, for every $b\in \sum\K(A)^2$
  we have: $$a^{2}+b\in I\,\,\Rightarrow\,\,a\in I$$
\end{defn}

\begin{rem}
Clearly,  $I$ is central $\Rightarrow$ $I$ is real $\Rightarrow$ $I$ is
  radical.
\end{rem}

\begin{rem}
An ideal $I\subset A$ is central if and only if $I$ is $(\sum\K(A)^2\cap A)$-radical in the sense of \cite{BCR}. 
\end{rem}

\begin{defn}
  \label{defcentralprime}
  \begin{enumerate}
\item[1)]
  We denote by $\RCent A$ the set of central prime ideals of $A$.
\item[2)] We denote by $\CentMax A$ the set of central and maximal ideals of $A$.
\item[3)] We say that $A$ is a central ring if any real ideal is central.
\item[4)] Assume $X$ is an irreducible affine algebraic
variety over $\R$ and let $\R[X]$ be the coordinate ring of $X$. We
denote by $\Cent X$ the image of $\CentMax \R[X]$ by the bijection $\ReMax \R[X]\to X(\R)$. We call $\Cent X$ the 
set of central real closed
points of $X$.
We say that $X$ is central if $X(\R)=\Cent X$.
\end{enumerate}
\end{defn}

Let $X$ be an affine algebraic variety over $\R$. We
recall classical notations. We denote by $X_{reg}(\R)$ the set of
smooth points of $X(\R)$. Let $W\subset X(\R)$, we denote by
$\overline{W}^Z$ (resp. $\overline{W}^{E}$) the closure of $W$ for the
Zariski (resp. Euclidean) topology.

Our definition of central ideals is chosen in order to satisfy
the Central Nullstellensatz stated in \cite[Cor. 7.6.6]{BCR}.
\begin{thm} (Central Nullstellensatz)\\
  \label{centralNull}
  Let $X$ be an irreducible affine algebraic variety over
  $\R$. Then:
  $$I\subset\R[X]\,\,{\rm is\,\, a\,\, central\,\, ideal}\,\, \Leftrightarrow\,\, I=\I(\Z(I)\cap
  \overline{X_{reg}(\R)}^E)\,\, \Leftrightarrow\,\, I=\I(\V(I)\cap
  \overline{X_{reg}(\R)}^E)$$
\end{thm}

\begin{proof}
  We assume $ I=\I(\Z(I)\cap
  \overline{X_{reg}(\R)}^E)$ and $a^{2}+b\in I$ for $a\in \R[X]$ and
  $b\in \sum\K(X)^2$. Since
  $a^{2}+b\in I$ then $a\in \I(\Z(I)\cap
  \overline{X_{reg}(\R)}^E)$ by \cite[Cor. 7.6.6]{BCR}. By hypothesis we get $a\in
  I$ and thus $I$ is central. 

  Assume $I$ is a central ideal of $\R[X]$. Let $a\in \I(\Z(I)\cap
  \overline{X_{reg}(\R)}^E)$. By \cite[Cor. 7.6.6]{BCR}, there exist $m\in\N$,
  $b\in \sum\K(X)^2$ such that $a^{2m}+b\in
  I$. Since $I$ central then $a^m\in I$. Since $I$ is radical then it
  follows that $a\in I$.
  
 To end the proof remark that $\Z(I)\cap  \overline{X_{reg}(\R)}^E=\V(I)\cap  \overline{X_{reg}(\R)}^E$.
\end{proof}

From the previous theorem, it follows:
\begin{cor} 
  \label{centgeo}
  Let $X$ be an irreducible affine algebraic variety over
  $\R$. Then:
  $$\Cent X=\overline{X_{reg}(\R)}^E.$$
\end{cor}

\begin{cor}
  \label{equivcentral}
  Let $X$ be an irreducible affine algebraic variety over
  $\R$. The ring $\R[X]$ is central if and only if $X$ is central.
\end{cor}

\begin{proof}
It follows from Theorems
\ref{RealNull} and \ref{centralNull}, that any real ideal of $\R[X]$
is central if and only if $\Cent
X=\overline{X_{reg}(\R)}^E=X(\R)$.
\end{proof}

We prove that we recover the definition of central prime ideal given
in \cite{FMQ-futur2}.
\begin{prop}
  \label{equivcentralprime}
Let $X$ be an irreducible affine algebraic variety over $\R$. Let
$\p\in\Sp \R[X]$. The
following properties are equivalent:
\begin{enumerate}
\item[1)] $\p\in\RCent \R[X]$.
\item[2)] $\overline{\Z(\p)\cap \Cent X}^{Z}=\Z(\p)$.
\item[3)] $\overline{\V(\p)\cap \Cent X}^{Z}=\V(\p)$.
\item[4)] $\p=\I(\Z(\p)\cap \Cent X)=\I(\V(\p)\cap \Cent X)$.
\item[5)] $\p=\{f\in\R[X]|\,\exists m\in\N,\,\exists g\in
  \sum\K(X)^2\,\,{\rm such\,\,that}\,\,f^{2m}+g\in\p\}$.
\item[6)] There exists an $\alpha\in\Sp_r(\K(X)\cap\R[X])$ which specializes in a
  order $\beta$ with support $\p$ i.e $\beta$ is in the closure of the
  singleton $\{\alpha\}$ for
  the topology of $\Sp_r \R[X]$.
\end{enumerate}
\end{prop}

\begin{proof}
The equivalence between 1) and 4) is given by
Theorem \ref{centralNull}.

The equivalence between 2) and 6) is \cite[Lem. 2.9]{FMQ-futur}.

Let us prove that 4), 3) and 2) are equivalent. Remark that we always
have $\p\subset \I(\Z(\p))\subset\I(\Z(\p)\cap \Cent X)$ and $\p\subset \I(\V(\p))\subset\I(\V(\p)\cap \Cent X)$ . Thus if
we assume that $\p=\I(\Z(\p)\cap \Cent X)$ (resp. $\p=\I(\V(\p)\cap \Cent X)$ ) then $\I(\Z(\p))=\p$ (resp. $\I(\V(\p))=\p$)
and it follows that $\p$ is a real ideal by the real Nullstellensatz (resp. $\p$ is a radical ideal by the classical Nullstellensatz),
moreover $\I(\Z(\p))=\I(\Z(\p)\cap \Cent X)$ (resp. $\I(\V(\p))=\I(\V(\p)\cap \Cent X)$) and it says that
$\overline{\Z(\p)\cap \Cent X}^{Z}=\Z(\p)$ (resp. $\overline{\V(\p)\cap \Cent X}^{Z}=\V(\p)$). We have proved 4)
implies 2) and 3).

Assume $\overline{\Z(\p)\cap \Cent X}^{Z}=\Z(\p)$ (resp. $\overline{\V(\p)\cap \Cent X}^{Z}=\V(\p)$) and
let $f\in\R[X]$. It follows that $\Z(\p)\subset\Z(f)$ (resp. $\V(\p)\subset\V(f)$) if and only if
$(\Z(\p)\cap \Cent X)\subset\Z(f)$ (resp. $(\V(\p)\cap \Cent X)\subset\V(f)$) and thus we get 4).

Clearly 5) implies 1). Assume $\p$ is central. Let $f\in\R[X]$ such that there exist $m\in\N$ and $g\in
 \sum \K(X)^2$ such that $f^{2m}+g\in\p$. Then $f^m\in\p$
  and thus $f\in\p$ since $\p$ is radical, it proves that 1) implies 5).
\end{proof}

\begin{ex}
Let $X$ be the Whitney umbrella i.e the real algebraic
surface with equation $y^2=zx^2$. Then $\p=(x,y)\subset \R[X]$ is a
central prime ideal since $\Z(\p)$ (the ``$z$''-axis and the stick of
the umbrella) meets $\Cent X$ in dimension one (the intersection
is half of the stick).
\end{ex}

\begin{ex} Let $X$ be the Cartan umbrella i.e the real algebraic
surface with equation $x^3=z(x^2+y^2)$. Then $\p=(x,y)\subset \R[X]$
is a real prime ideal but not a
central ideal by 2) of Proposition \ref{equivcentralprime} since $\Z(\p)$ (the ``$z$''-axis and the stick of
the umbrella) meets $\Cent X$ in a single point. We prove now directly that $\p$ is not central:\\
We have $$b=x^2+y^2-z^2=x^2+y^2-\frac{x^6}{(x^2+y^2)^2}=\frac{3x^4y^2+3x^2y^4+y^6}{(x^2+y^2)^2}\in (\sum\K(X)^2)\cap \R[X]$$
thus $z^2+b=x^2+y^2\in\p$ but $z\not\in\p$.
\end{ex}

We give a central version of \cite[Lem. 4.1.5]{BCR}.
\begin{prop}
  \label{centralnoeth}
Assume $A$ is noetherian. If $I\subset A$ is a central ideal
then the minimal prime ideals containing $I$ are central ideals.
\end{prop}

\begin{proof}
Let $\p_1,\ldots\p_l$ be the minimal prime ideals containing $I$. If
$l=1$ then $I=p_1$ since $I$ is radical and thus the proof is done in
that case. So we assume $l>1$ and $\p_1$ is not central. There exist
$a\in A\setminus \p_1$, $b_1,\ldots,b_k\in \K(A)$ such that $a^{2}+b_1^2+\cdots+b_k^2\in
  \p_1$. We choose $a_2,\ldots,a_l$ such that $a_i\in
  \p_i\setminus\p_1$ and we set $c=\prod_{i=2}^l a_i$. Then
  $(ac)^2+(b_1c)^2+\cdots+(b_lc)^2\in\bigcap_{i=1,\ldots,l}\p_i=I$
  ($I$ is radical). Thus $ac\in\p_1$, a contradiction.
\end{proof}

\begin{defn}
  \label{defradcent}
  Let $I\subset A$ be an ideal. We define the
  central radical of $I$, denoted by $\RadCe I$, as follows:
  $$\RadCe I=\{a\in A|\,\, \exists m\in\N \,\, \exists b\in \sum\K(A)^2\,\, {\rm such\,\,that}\,\,a^{2m}+b\in I
  \}.$$
\end{defn}

We give a central version of Proposition \ref{propradreal}. It can be derived from \cite[Prop. 4.2.6]{BCR} using the theory of convex ideals for a cone.
\begin{prop}
\label{propradcent}
Let $I\subset A$ be an ideal. We have:
\begin{enumerate}
\item[1)] $\RadCe I$ is the smallest central ideal of $A$ containing
  $I$.
\item[2)] $$\RadCe I=\bigcap_{\p\in\RCent A,\,I\subset\p}\p.$$
\end{enumerate}
\end{prop}

\begin{proof}
We show that $\RadCe I$ is an ideal. It is clear that $0\in\RadCe
I$. Let $a\in\RadCe I$. There exist $m\in\N$, $b_1,\ldots,b_k\in \K(A)$
such that $a^{2m}+b_1^2+\cdots+b_k^2\in I$. Let $a'\in A$. Since
$(aa')^{2m}+(b_1(a')^m)^2+\cdots+(b_k(a')^m)^2\in I$ and since $b_i(a')^m\in \K(A)$
then $aa'\in\RadCe I$. To show that $\RadCe I$ is closed under
addition then copy the proof of \cite[Prop. 4.1.7]{BCR} with the
conditions that the $b_i$ and $b'_j$ are only in $\K(A)$ rather than in
$A$.

We show that $\RadCe I$ is a central ideal. Let $a\in A$ and $b_1,\ldots,b_k\in \K(A)$ such that $a^{2}+b_1^2+\cdots+b_k^2\in \RadCe I$. Thus there exist
$m\in\N$ and $c_1,\ldots,c_l\in \K(A)$ such that $(a^{2}+b_1^2+\cdots+b_k^2)^{2m}+c_1^2+\cdots+c_l^2\in I$. It follows that there exist 
$d_1,\ldots,d_t\in\K(A)$ such that $a^{4m}+d_1^2+\cdots+d_t^2\in I$ and thus $a\in\RadCe I$.

Let $J$ be a central ideal of $A$ containing $I$. Let $a\in \RadCe
I$. There exist $m\in\N$, $b\in \sum\K(A)^2$
such that $a^{2m}+b\in J$. Thus $a^m\in J$ by
centrality of $J$ and finally $a\in J$ by radicality of $J$. The proof
of 1) is done.

We denote by $I'$ the ideal $$\bigcap_{\p\in\RCent
  A,\,I\subset\p}\p.$$ From 1) we get $\RadCe I\subset I'$. Let us
show the converse inclusion. Let $a\in A\setminus\RadCe I$. Let $J$ be
maximal among the central ideals containing $I$ but not $a$. If $J$ is
not prime then, following the proof of \cite[Prop. 4.1.7]{BCR}, we can
find $m\in\N$, $b\in \sum\K(A)^2$
such that $a^{2m}+b\in I$, it gives a
contradiction. Hence $J$ is a prime ideal and thus $a\not\in I'$.
\end{proof}

\begin{cor}
  \label{corradcent}
Let $I\subset A$ be an ideal. Then, $I$ is a
central ideal if and only if $I=\RadCe I$.
\end{cor}

To end this section, we study the existence of a central ideal.
\begin{prop}
\label{existcentralideal}
The following properties are equivalent:
\begin{enumerate}
\item[1)] $A$ is a real domain.
\item[2)] $\RCent A\not=\emptyset$.
\item[3)] $A$ has a proper central ideal.
\item[4)] $(0)$ is a central ideal of $A$.
\end{enumerate}
Assume $A$ is the coordinate ring of an irreducible affine algebraic variety over $\R$. Then the previous properties are equivalent to the two following ones:
\begin{enumerate}
\item[5)] $X_{reg}(\R)\not=\emptyset$.
\item[6)] $X(\R)$ is Zariski dense in $X(\C)$ and $\Sp A$.
\end{enumerate}
\end{prop}

\begin{proof}
It is clear that 4) implies 2) and 3). By Proposition \ref{defrealdomain} then 4) implies 1).
Assume $A$ is a real domain. By Proposition \ref{defrealdomain}, we know that $(0)$ is a real ideal and we will prove that it is moreover a central ideal. Assume $a^2+b=0$ with $a\in A$ and $b\in \sum \K(A)^2$. It gives an identity $c^2a^2+s=0$ with $c\in A\setminus\{0\}$ and $s\in\sum A^2$. Since $(0)$ is a real ideal then it follows $a=0$. We get 1) implies 4). Since a prime ideal is proper then 2) implies 3). Assume $I\subset A$ is a proper central ideal. By Corollary \ref{corradcent} we have $I=\RadCe I$. By Proposition \ref{propradcent}, $I$ is the intersection of the central prime ideals of $A$ containing $I$, it follows that the set of central prime ideals of $A$ containing $I$ is non-empty and 3) implies 2). Let $I\subset A$ be a proper and central ideal of $A$. Assume $A$ is not a real domain. By Proposition \ref{defrealdomain}, we get that $-1\in\sum\K(A)^2$ and since $1^2+(-1)=0\in I$ and $I$ is central then it follows that $1\in I$, impossible.  Thus 3) implies 1).

Assume $A$ is the coordinate ring of an irreducible affine algebraic variety $X$ over $\R$. Assume $A$ is a real domain. We have proved that it implies $(0)$ is a central ideal. 
By 3) of Proposition \ref{equivcentralprime}, it follows that $\Cent X$ is Zariski dense in $\Sp A$. Hence $X(\R)$ is also Zariski dense in $\Sp A$ (and in $X(\C)$). It proves 
that 1) implies 6).  If $X_{reg}(\R)\not=\emptyset$ then $\Cent X\not=\emptyset$ and thus 5) implies 2). Assume $X(\R)$ is Zariski dense in $\Sp A$ then it intersects the set of regular prime ideals of $A$ which is a non-empty Zariski open subset of $\Sp A$ and thus 6) implies 5).
\end{proof}

\section{Integral extensions and lying-over properties}

In the sequel we consider rings up to isomorphisms and affine algebraic varieties up to isomorphisms. In particular, when we write an equality of rings it means they are isomorphic, the reader should remember this especially when speaking about uniqueness.

\subsection{Integral extensions and normalization}

Let $A\to B$ be an extension of domains. The extension is said of finite type (resp. finite) if it makes $B$ a finitely generated $A$-algebra (resp. $A$-module). We say that $A\to B$ is
birational if it induces an isomorphism between the fraction fields
$\K(A)$ and $\K(B)$. We say that an
element $b\in B$ is integral over $A$ if $b$ is the root of a monic
polynomial with coefficients in $A$. By \cite[Prop. 5.1]{AM}, $b$ is
integral over $A$ if and only if $A[b]$ is a finite $A$-module. This
equivalence allows to prove that $A_B'=\{b\in B|\,b\,\, {\rm is\,\,
  integral\,\, over}\,\,A\}$ is a ring called the integral closure of $A$ in
$B$. The extension $A\to B$ is said to be integral if $A_B'=B$. In
case $B=\K(A)$ then the ring $A_{\K(A)}'$ is
denoted by $A'$ and is simply called the integral closure of $A$.
The ring $A$ is called integrally closed (in $B$) if
$A=A'$ ($A=A_B'$).
If $A$ is the
coordinate ring of an irreducible affine algebraic variety $X$ over a
field $k$ then $A'$ is a
finite $A$-module (a theorem of Emmy Noether \cite[Thm. 4.14]{Ei}) and thus it is a finitely generated $k$-algebra and
so $A'$ is the coordinate ring of an irreducible affine algebraic
variety over $k$, denoted by $X'$, called the normalization
of $X$. For a morphism $\pi:X\rightarrow Y$ between two affine 
algebraic varieties over a field $k$, we denote by $\pi^*:k[Y]\to k[X]$, $f\mapsto f\circ\pi$ the associated ring morphism.
We recall that a morphism $X\rightarrow Y$ between two irreducible affine 
algebraic varieties over a field $k$ is said of finite type (resp. finite) (resp. birational) if the ring morphism
$k[Y]\rightarrow k[X]$ is of finite type (resp. finite) (resp. birational).
The inclusion $k[X]\subset k[X']=k[X]^\prime$ induces a finite and birational morphism
which we denote by $\pi':X'\rightarrow X$, called the normalization
morphism. We say that an irreducible affine 
algebraic variety $X$ over a field $k$ 
is normal if its coordinate ring is integrally
closed.

\subsection{Contraction and lying-over properties}

For an extension of rings, it is clear that the contraction of a real ideal is a real ideal.
We prove that, for an extension of domains, the contraction of a central ideal remains central.
\begin{prop}
\label{contractcentral}
Let $A\to B$ be an extension of domains. If $I$ is a central ideal of $B$ then $I\cap A$ is a central ideal of $A$. In particular, the map
$$\RCent B\to\RCent A,\,\,\q\mapsto \q\cap A$$ is well defined.
\end{prop}

\begin{proof}
The proof is clear since $\sum\K(A)^2\subset\sum\K(B)^2$.
\end{proof}

\begin{rem}
As noticed in \cite{FMQ-futur2} the result of the previous proposition cannot be generalized in the reducible case (even for extensions of reduced rings with a finite number of minimal prime ideals that are real) and it is the reason we restrict ourself to extension of domains in the paper. There are some problems if for example the contraction of a minimal prime ideal of $B$ is not a minimal prime ideal of $A$. Consider the extension 
$$A=\R[C]=\R[x,y]/(y^2-x^2(x-1))\to B=\R[C]\times (\R[C]/(x,y)), \,\, f\mapsto (f,f(0,0))$$
The extension $A\to B$ is associated to the morphism of affine algebraic varieties $C'\to C$ where $C$ is the plane cubic with a real isolated point, $C'$ is the disjoint union of $C$ and a real point, the morphism is the identity on $C$ and maps the point onto the origin.  The contraction to $A$ of the minimal and central prime ideal $\R[C]\times (0)$ (central here means central in its irreducible component) of $B$ is the real maximal ideal corresponding to the isolated real point and thus the contracted ideal is not central (Proposition \ref{equivcentralprime}).
\end{rem}

We have the following lying over properties:
\begin{prop}
  \label{lying-over}
Let $A\to B$ be an integral extension of domains. Then:
\begin{itemize}
\item[1)] $\Sp B\to\Sp A$, $\q\mapsto \q\cap A$ is surjective.
\item[2)] $\Max B\to\Max A$ is well defined and surjective.
\item[3)] If $A\to B$ is birational then the map
    $\RCent B\to\RCent A$ is surjective.
\item[4)] If $A\to B$ is birational then the map
    $\CentMax B\to\CentMax A$ is well defined and surjective.
\end{itemize}
\end{prop}

\begin{proof}
See \cite[Thm. 9.3]{Ma} for statements 1) and 2). From Proposition \ref{contractcentral} and \cite[Prop. 2.8]{FMQ-futur2} we get 3) in the case $A$ is a real domain. 
Assume $A$ is a domain but $\K(A)$ is not real and $A\to B$ is integral and birational. Then $\K(B)$ is not real and by Proposition \ref{existcentralideal}
then $\RCent A=\RCent B=\emptyset$ and we get 3) in that case.
Statement 4) is a consequence of 2) and 3).
\end{proof}

\begin{rem}
We do not have a lying over property for real prime ideals even for birational extensions. Consider for example the integral and birational extension 
$A=\R[x,y]/(y^2-x^2(x-1))\to\R[x,Y]/(Y^2-(x-1))=B$ given by $x\mapsto x$ and $y\mapsto Yx$. The extension is integral and birational since it corresponds to the normalization of 
the plane cubic curve with a real isolated node given by the equation $y^2-x^2(x-1)=0$ and thus $B=A'$. Over the real but not central ideal $(x,y)$ in $A$ there is a unique ideal of $B$
given by $(x,Y^2+1)$ and this ideal is not real.
This is the principal reason we work here with the central spectrum rather than the real spectrum. 
\end{rem}

\begin{rem}
Consider for example the integral extension $A=\R[x]\to\R[x,y]/(y^2-x)=B$, we do not have any real prime ideal of $B$ lying over the real and central prime ideal $(x+1)$ of $A$.
This example shows that we do not have a central lying over property for integral extensions of domains which are not birational. From \cite{FMQ-futur2}, the central lying-over property exists more generally for an integral extension $A\to B$ of domains such that 
$\Sp_r \K(B)\to\Sp_r \K(A)$ is surjective.
\end{rem}

\section{Central seminormalization for rings}

\subsection{Centrally subintegral extension}

Recall (\cite{V}) that an extension $A\to B$ is said subintegral if it
is an integral extension, for
any prime ideal $\p\in\Sp A$ there exists a unique
prime ideal $\q\in\Sp B$ lying over $\p$ (it means $\Sp B\to\Sp A$ is bijective), and furthermore for any such
pair $\p$, $\q$ the induced injective map on
the residue fields $k(\p)\rightarrow k(\q)$ is an isomorphism. To
characterize the last property, we say that $\Sp B\to\Sp A$ is
equiresidual. In summary an integral extension $A\to B$ is subintegral if and only if $\Sp
B\to\Sp A$ is bijective and equiresidual. Such a concept is related to the notion of "radiciel" morphism of schemes introduced by Grothendieck \cite[I Def. 3.7.2]{Gr1}.
Swan gave another nice characterization of a subintegral extension \cite[Lem. 2.1]{Sw}: an extension $A\to B$ is subintegral if $B$ is integral over $A$ and for all morphisms $A\to K$ into a field $K$, there exists a unique extension $B\to K$.

In the same spirit, we can give a natural definition of a central
subintegral extension.

\begin{defn}
\label{defcentsubint}
Let $A\to B$ be an extension of domains.
We say that $A\to B$ is centrally subintegral or $s_c$-subintegral for short (we follow the notation used in \cite{FMQ-futur2})
  if it is an integral extension, and if the map $\RCent B\to\RCent A$ is bijective
  and equiresidual.
\end{defn}

\begin{rem}
\label{triviallyempty}
From Propositions \ref{existcentralideal} and \ref{contractcentral}, any integral extension $A\to B$ of a non-real domain $A$ is trivially centrally subintegral since $\RCent A=\RCent B=\emptyset$.
\end{rem}

\begin{rem}
\label{remscbirat}
Let $A\to B$ be a centrally subintegral extension of domains and assume $A$ is real. By 4) of Proposition \ref{existcentralideal} then $(0)$ is a central ideal of $A$.
Since the null ideal of $B$ is the unique prime ideal of $B$ lying over the null ideal of $A$ then, by bijectivity of the central spectra, $(0)$ is also a central ideal of $B$. By equiresiduality then the extension $A\to B$ is birational.
\end{rem}

\begin{ex} The finite extension $A=\R[x]\to\R[x,y]/(y^2-x^3)=B$ satisfies the property that $\CentMax B\to\CentMax A$ is bijective and equiresidual but $A\to B$ is not birational
and so $A\to B$ is not centrally subintegral.
\end{ex}

\begin{ex} The finite extension $\R[x,y]/(y^2-x^3)\to \R[t]$ given by $x\mapsto t^2$ and $y\mapsto t^3$ (corresponding to the normalization of the cuspidal curve) is centrally subintegral.
\end{ex}

From \cite[Lem. 2.1]{Sw} we derive another characterization of a centrally subintegral extension.
\begin{prop}
\label{Swan}
An extension $A\to B$ is centrally subintegral if $B$ is integral over $A$ and for all morphisms $\varphi:A\to K$ into a field $K$ with $\ker\varphi\in\RCent A$, there exists a unique extension $\psi:B\to K$ such that $\ker\psi\in\RCent B$.
\end{prop}

We want now to characterize differently these centrally subintegral extensions in the case
we work with geometric rings.

Let $X$ be an irreducible affine algebraic variety over
$\R$. A (resp. irreducible) real algebraic subvariety $V$ of $X$ is a
closed Zariski subset of $\Sp \R[X]$ of the form 
$V=\V(I)=\{\p\in\Sp \R[X]|\,\,I\subset \p\}\simeq \Sp(\R[X]/I)$ for $I$ an
(resp. prime) ideal of $\R[X]$. In that case the real part of
$V$, denoted by $V(\R)$, is the closed Zariski subset of $X(\R)$ given
by $\Z(I)$. An algebraic subvariety $V$ of $X$ is said central in $X$ if 
$V=\V(I)\simeq \Sp(\R[X]/I)$ for $I$ a central ideal in $\R [X]$. By Theorem
\ref{centralNull}, an irreducible real algebraic subvariety $V$ of $X$ is central in $X$
if and only if $\overline{V(\R)\cap \Cent X}^Z=V(\R)$.

\begin{rem} The stick is central in the Whitney umbrella but it is not the case
 in the Cartan umbrella.
\end{rem}

\begin{rem} For an irreducible real algebraic subvariety $V$ of $X$, the properties "$V$ is central" and "$V$ is central in $X$" are distinct. As example, take the stick of the Cartan umbrella.
\end{rem}

\begin{defn}
Let $\pi:Y\to X$ be a dominant morphism between
irreducible affine algebraic varieties over $\R$. We say that $\pi$ is centrally subintegral or $s_c$-subintegral if the extension $\pi^*:\R[X]\to \R[Y]$ is $s_c$-subintegral.
\end{defn}

Let $\pi:Y\to X$ be a dominant morphism between
irreducible affine algebraic varieties over $\R$. By
Proposition \ref{contractcentral}, we have an associated map
$\RCent \R[Y]\to\RCent \R[X]$. We say that $\pi:Y\to X$ is centrally hereditarily birational
if for any irreducible real algebraic subvariety $V=\V(\p)\simeq\Sp(\R[Y]/\p)$ central in $Y$, the
morphism $\pi_{|V}: V\to W=\V(\p\cap \R[X])\simeq\Sp(\R[X]/(\p\cap\R[X]))$ is 
birational i.e the extension $k(\p\cap \R[X])=\K(W)\to k(\p)=\K(V)$ is an isomorphism. By Proposition \ref{existcentralideal} a centrally hereditarily birational morphism 
$Y\to X$ is birational if $X_{reg}(\R)\not=\emptyset$.
From above remarks we easily get:
\begin{prop}
 \label{heredbirat=equiresid}
Let $\pi:Y\to X$ be a dominant morphism between
irreducible affine algebraic varieties over $\R$. The following properties are equivalent:
\begin{enumerate}
 \item[1)] The morphism $\pi:Y\to X$ is centrally hereditarily birational.
 \item[2)] The map $\RCent \R[Y]\to \RCent \R[X]$ is equiresidual.
\end{enumerate}
\end{prop}

From Proposition \ref{lying-over}, with an additional finiteness hypothesis we get:
\begin{cor}
 \label{corheredbirat=equiresid}
Let $\pi:Y\to X$ be a finite morphism between
irreducible affine algebraic varieties over $\R$. The following properties are equivalent:
\begin{enumerate}
 \item[1)] The morphism $\pi:Y\to X$ is centrally hereditarily birational and the map $\RCent \R[Y]\to\RCent \R[X]$ is bijective.
 \item[2)] $\pi$ is $s_c$-subintegral.
\end{enumerate}
\end{cor}

Let $X$ be an irreducible affine algebraic variety over
$\R$ such that $X_{reg}(\R)\not=\emptyset$. 
Following \cite{FMQ-futur2}, we denote by $\SR(\Cent X)$, called
the ring of rational continuous functions on $\Cent X$, the ring
of continuous functions on $\Cent X$ that are rational on $X$ i.e coincide with a regular
function on a non-empty Zariski open subset of $X(\R)$ intersected
with $\Cent X$. We denote by $\SRR (\Cent X)$, called the ring
of regulous functions on $\Cent X$, the subring of $\SR(\Cent X)$
given by rational continuous functions $f\in\SR(\Cent X)$ that satisfies the additional
property that for any irreducible real algebraic subvariety $V=\V(\p)$ for $\p\in\RCent \R[X]$ of
$X$ then the restriction of $f$ to $V(\R)\cap \Cent X$ is
rational on $V$ i.e lies in $k(\p)$. Remark that for a variety with at least a smooth real closed point,
then being rational, rational on the real closed points or rational on the central closed points is the same (Proposition \ref{existcentralideal}).

Let $\pi:Y\to X$ be a finite and birational morphism between
irreducible affine algebraic varieties over $\R$. By
Proposition \ref{lying-over}, we have associated surjective maps
$\RCent \R[Y]\to\RCent \R[X]$ and $\Cent Y\to\Cent X$. The
composition by $\pi$ induces natural morphisms $\SR(\Cent
X)\to\SR(\Cent Y)$ and $\SRR(\Cent
X)\to\SRR(\Cent Y)$. We say that the map $\Cent Y\to\Cent
X$ is biregulous if it is bijective and the inverse
bijection is a regulous map i.e its component are in $\SRR(\Cent
X)$. Such a concept is related to Grothendieck's notion of universal homeomorphism between schemes \cite[I 3.8]{Gr1}.

The following theorem from \cite{FMQ-futur2} explains how $s_c$-subintegral extensions and regulous functions are related.
\begin{thm} \cite[Lem. 3.13, Thm. 3.16]{FMQ-futur2}\\
  \label{scgeom}
Let $\pi:Y\to X$ be a finite and birational morphism between
irreducible affine algebraic varieties over $\R$. The following
properties are equivalent:
\begin{enumerate}
\item[1)] $\pi$ is $s_c$-subintegral.
\item[2)] The morphism $\pi:Y\to X$ is centrally hereditarily birational and the map
  $\RCent \R[Y]\to\RCent \R[X]$ is bijective.
\item[3)] $\SRR(\Cent
  )\to\SRR(\Cent Y)$, $f\mapsto f\circ\pi_{|\Cent Y}$ is an isomorphism.
\item[4)] The map $\pi_{|\Cent Y}:\Cent Y\to\Cent
  X$ is biregulous.
\item[5)] For all $g\in\R[Y]$ there exists $f\in\SRR(\Cent X)$ such that $f\circ\pi_{|\Cent Y}=g$ on $\Cent Y$.
\end{enumerate}
\end{thm}

\begin{rem}
All these equivalent properties are trivially satisfied if $\Cent X=\emptyset$ (Remark \ref{triviallyempty}).
\end{rem}

\subsection{Classical algebraic seminormalization}

We recall in this section the principal result obtained by Traverso \cite{T} concerning the seminormality of a ring in another one.

\begin{defn}
\label{intermedring}
A ring $C$ is said intermediate between the rings $A$ and $B$ if there exists a sequence of extensions $A\to C\to B$. In that case, we say that $A\to C$ and $C\to B$ are intermediate extensions of $A\to B$ and we say moreover that $A\to C$ is a subextension of $A\to B$.
\end{defn}

Seminormal extensions are maximal subintegral extensions.
\begin{defn}
\label{defseminorm}
Let $A\rightarrow C\to B$ be a sequence of two extensions of rings.
We say that $C$ is seminormal between $A$ and $B$ 
  if $A\to C$ is subintegral
  and moreover if for every
  intermediate domain
  $D$ between $C$ and $B$ with $C\subsetneq D$ then $A\to D$ is not
  subintegral. We say that $A$ is seminormal in $B$ if $A$ is
  seminormal between $A$ and $B$.
\end{defn}

\begin{defn}
Let $A$ be a ring and let $I$ be an ideal of $A$.
The Jacobson radical of $A$, denoted by $\JRad(A)$, is the
intersection of the maximal ideals of $A$.
\end{defn}

For a given extension of rings $A\to B$, Traverso (see \cite{T} or
\cite{V}) proved there exists a unique intermediate ring which is seminormal between $A$ and $B$.
\begin{thm} \label{existseminorm}
Let $A\to B$ be an extension of rings. There exists a unique ring between $A$
and $B$ which is seminormal between $A$ and $B$, this ring is denoted
by $A_B^*$, it is called the seminormalization of $A$ in $B$ and
moreover we have $$A_B^*=\{b\in A_B'|\,\,\forall\p\in\Sp
A,\,\,b_{\p}\in A_{\p}+\JRad((A_B')_{\p})\}.$$
\end{thm}

Remark that to build $A_B^*=\{b\in B|\,\,\forall\p\in\Sp
A,\,\,b_{\p}\in A_{\p}+\JRad((A_B')_{\p})\}$ then, for all $\p\in\Sp A$, we glue together all the
prime ideals of $A_B'$ lying over $\p$.

\subsection{Introduction to the Central algebraic Seminormalization Existence Problem}

Intermediate extensions of a centrally subintegral extension are still centrally subintegral extensions:
\begin{lem}
  \label{propcentsubint}
Let $A\to C\to B$ be a sequence of extensions of 
domains. Then $A\to B$ is $s_c$-subintegral if and only if $A\to C$
and $C\to B$ are both $s_c$-subintegral.
\end{lem}

\begin{proof}
Assume $A\to B$ is $s_c$-subintegral. Clearly, $A\to C$ and $B\to C$
are both integral extensions. 
It follows that $\RCent B\to
\RCent A$ is bijective and equiresidual. If $A$ is not a real domain then it follows from Remark \ref{triviallyempty} that $A\to C$ and $C\to B$ are trivially $s_c$-subintegral.
Assume now $A$ is a real domain. By equiresiduality (Remark \ref{remscbirat}) then $A\to
B$ is birational and thus $A\to C$ and $C\to B$ are also both
birational. By Proposition \ref{lying-over} the maps $\RCent B\to
\RCent C$ and $\RCent C\to\RCent A$ are surjective, and since the
composition is bijective then they are both bijective. Let
$\q\in\RCent B$ then we have the following sequence of extensions of
residue fields $k(\q\cap A)\to k(\q\cap C)\to k(\q)$, it shows that $\RCent B\to
\RCent C$ and $\RCent C\to\RCent A$ are both equiresidual.

The converse implication is clear.
\end{proof}

By Lemma \ref{propcentsubint} a subextension of a centrally
subintegral extension is still centrally subintegral, so we may consider
maximal centrally subintegral subextensions.
\begin{defn}
\label{defscnorm}
Let $A\rightarrow C\to B$ be a sequence of two extensions of
domains.
We say that $C$ is centrally seminormal (or $s_c$-normal for short) between $A$ and $B$ 
  if $A\to C$ is $s_c$-subintegral
  and moreover if for every
  intermediate domain
  $C'$ between $C$ and $B$ with $C\not= C'$ then $A\to C'$ is not
  $s_c$-subintegral. We say that $A$ is $s_c$-normal in $B$ if $A$ is
  $s_c$-normal between $A$ and $B$.
\end{defn}

From Lemma \ref{propcentsubint}, we get an equivalent definition of a
centrally seminormal ring (between $A$ and $B$):
\begin{prop}
  \label{defequivscnorm}
Let $A\rightarrow C\to B$ be a sequence of two extensions of
domains. Then, $C$ is
$s_c$-normal between $A$ and $B$ if and only $A\to C$ is $s_c$-subintegral and $C$ is
$s_c$-normal in $B$.
\end{prop}

From Definition \ref{defscnorm} we easily deduce the following
property:
\begin{prop}
  \label{intermednorm}
Let $A\rightarrow C\to B$ be a sequence of two extensions of
domains. If $A$ is $s_c$-normal in $B$ then $A$ is $s_c$-normal in
$C$.
\end{prop}

\begin{defn}
Let $A$ be a ring and let $I$ be an ideal of $A$.
\begin{enumerate}
\item The real Jacobson radical of $A$, denoted by $\JRadRe(A)$, is the
intersection of the maximal and real ideals of $A$.

\item The central Jacobson radical of $A$, denoted by $\JRadCe(A)$, is the
  intersection of the maximal and central ideals of $A$.
\end{enumerate}
\end{defn}

In view to the classical case (see the previous section), we state the following problem:\\
Given an extension $A\to B$ of domains, is there a 
unique intermediate domain $C$ which is $s_c$-normal between $A$ and $B$?\\
\\
We define the central seminormalization (or $s_c$-normalization) of
$A$ in $B$ as the
ring which would give a solution to this problem. In the classical case, the problem is solved by Theorem \ref{existseminorm}.
\begin{defn}
  \label{defscnormal}
Let $A\to B$ be an extension of domains. In case there exists a unique maximal element among the
intermediate domains $C$ between $A$ and $B$ such that $A\to C$ is
$s_c$-subintegral then we denote it by $A^{s_c,*}_B$ and we call it the
central seminormalization or
$s_c$-normalization of $A$ in $B$. In case $B=A'$ then we omit $B$ and
we call $A^{s_c,*}$ the $s_c$-normalization of $A$.
\end{defn}

The existence of a central seminormalization is already proved in \cite{FMQ-futur2} in the special case $B=A'$.

\subsection{Central gluing over a ring}

In view of the classical case (see Theorem \ref{existseminorm}), a candidate to be the
$s_c$-normalization of $A$ in $B$ when $A\to B$ is integral is the following ring. 
\begin{defn}
  \label{defgluingAB}
Let $A\to B$ be an integral extension of domains. The ring $$A_B^{s_c}=\{b\in B|\,\,\forall\p\in\RCent
A,\,\,b_{\p}\in A_{\p}+\JRadCe(B_{\p})\}$$ is called the central gluing
of $B$ over $A$.
\end{defn}

The central gluing is not the $s_c$-normalization.
\begin{ex}\label{cex}
Consider the finite extension $A=\R[x]\to\R[x,y]/(y^2+x^2+1)=B$. Then $A_B^{s_c}=B$ since $\RCent B=\emptyset$ and $A\to B$ is not centrally subintegral.
\end{ex}

In the following, if $\ir$ is a prime
ideal of a ring $C$, we denote by $c(\ir)$ the class of $c\in C$ in
$k(\ir)$.

The central gluing satisfies the following universal property again related to the notion of "radiciel" morphism of schemes introduced by Grothendieck \cite[I Def. 3.7.2]{Gr1}:
\begin{thm}
  \label{PUgluingAB}
Let $A\to B$ be an integral extension of domains. The
central gluing $A^{s_c}_B$ of $B$ over $A$ is the biggest 
intermediate ring $C$ between $A$ and $B$ 
satisfying the following
properties:
\begin{enumerate}
\item If $\q_1,\q_2\in\RCent B$ ly over $\p\in\RCent A$ then $\q_1\cap
  C=\q_2\cap C$.
\item If $\q\in\RCent B$ then the residue fields extension $k(\q\cap
  A)\to k(\q\cap C)$ is an isomorphism.
\end{enumerate}
\end{thm}

\begin{proof}
We first prove that $A^{s_c}_B$ satisfies (1) and (2). Let
$\p\in\RCent A$ and let $\q_1,\q_2\in\RCent B$ lying over $\p$. Since
$\q_1B_{\p}$ and $\q_2B_{\p}$ are two maximal and central ideals of
$B_{\p}$ then, by
definition of $A^{s_c}_B$, we get $\q_1\cap
A^{s_c}_B=\q_2\cap A^{s_c}_B=(\p A_{\p}+\JRadCe(B_{\p}))\cap A^{s_c}_B$. Since
$k(\p)=A_{\p}/\p A_{\p}= (A^{s_c}_B)_{\p}/((\p A_{\p}+\JRadCe(B_{\p}))\cap
A^{s_c}_B)_{\p}$ then the first part of the proof is done.

To end the proof, it is sufficient to show that if $C$ is intermediate
between $A$ and $B$ and satisfies (1) and (2) 
then $C\subset A^{s_c}_B$. We have to show that if
$\p\in\RCent A$ then $C\subset (A_{\p}+\JRadCe(B_{\p}))$. If there is
no central prime ideal of $B$ lying over $\p$ then $C\subset
A_{\p}+\JRadCe(B_{\p})=B_{\p}$. Assume now there is at least one central
prime ideal of $B$, say $\q$, lying over $\p$. Since $C$ satisfies (1) then $\q\cap C$ is the unique central prime of $C$ lying
over $\p$ that is
the contraction of a central prime ideal of $B$. It follows that $(\q\cap
C)B_{\p}\subset \JRadCe(B_{\p})$. We use the following commutative diagram
\[
\begin{array}{ccc}
k(\p) & \simeq & k(\q\cap C)  \\
\uparrow &&\uparrow\\
A_{\p} &\rightarrow & C_{\p}
\end{array}
\]\\
Let $c\in C$. By (2) there exist $a\in A$ and $s\in A\setminus \p$ such that
$(a/s)(\q\cap C)=c(\q\cap C)$. Hence $a-sc\in\q\cap C$ and thus $c-a/s\in
(\q\cap C)C_{\p}=\ker (C_{\p}\to k(\q\cap C))$. We get $c\in A_{\p}+\JRadCe{B_{\p}}$.
This concludes the proof.
\end{proof}

For integral extensions, the central gluing contains every centrally subintegral subextensions.
\begin{cor}
  \label{inclnormglue}
Let $A\to C\to B$ be a sequence of integral extensions of
domains. If $A\to C$ is $s_c$-subintegral then $$C\subset A^{s_c}_B.$$
\end{cor}

\begin{proof}
If $A\to C$ is $s_c$-subintegral then it is easy to see that $C$
satisfies the properties (1) and (2) of Theorem \ref{PUgluingAB}. We
conclude by Theorem \ref{PUgluingAB}.
\end{proof}

\subsection{Birational closure}

\subsubsection{Definition}

\begin{defn} \label{fiberproduct}
Let $A\to B$ and $C\to B$ be two extensions of rings. The fibre product $A\times_B C$ is the ring defined by the following pull-back diagram
\[
\begin{array}{ccccc}
& A\times_B C & \to & A & \\
&\downarrow &&\downarrow & \\
& C &\to & B& \\
\end{array}
\]
\end{defn}

\begin{defn} \label{defbiratclos}
Let $A\to B$ be an extension of domains and let $\K(A)\to\K(B)$ be the associated extension of fields. We denote by $\widetilde{A}_B$ the fibre product $B\times_{\K(B)}\K(A)$ and we call it the birational closure of $A$ in $B$.
\end{defn}

The birational closure of $A$ in $B$ is the biggest intermediate ring between $A$ and $B$ which is birational with $A$.
\begin{prop} \label{PUbiratclos}
Let $A\to B$ be an extension of domains and let $\K(A)\to\K(B)$ be the associated extension of fields. Then, $\widetilde{A}_B$ is intermediate between $A$ and $B$ with $A\to \widetilde{A}_B$ birational and moreover if $C$ is an intermediate ring between $A$ and $B$ and $A\to C$ is birational then $C\to B$ factorizes uniquely through $\widetilde{A}_B$.
\end{prop}

\begin{proof}
The following commutative diagram \[
\begin{array}{ccccc}
& A & \to & B & \\
&\downarrow &&\downarrow & \\
& \K(A) &\to & \K(B)& \\
\end{array}
\]
gives a factorization of $A\to B$ through $\widetilde{A}_B$ by universal property of the fibre product. Since we have an extension $\widetilde{A}_B\to \K(A)$ then $\K(\widetilde{A}_B)=\K(A)$ and thus $A\to \widetilde{A}_B$ is birational.
Let $C$ be an intermediate domain between $A$ and $B$ such that $A\to C$ is birational. We get the following commutative diagram
\[
\begin{array}{cccccccc}
& A & \to & C& = & C &\to&B  \\
&\downarrow & &\downarrow & &&& \downarrow\\
& \K(A) &= & \K(C)&=& \K(A) & \to & \K(B) \\
\end{array}
\]
By universal property of the fibre product then the extension $C\to B$ factorizes uniquely through $\widetilde{A}_B$.
\end{proof}

\subsubsection{Integral and birational closure}

We prove that the operations "integral closure" and "birational closure" commute together.
\begin{prop} \label{commutclos}
Let $A\to B$ be an extension of domains. Then $$A_{\widetilde{A}_B}'=\widetilde{A}_{A_B '}$$
\end{prop}

\begin{proof}
The extension $A\to A_{\widetilde{A}_B}'$ is integral so $A\to A_B '$ factorizes uniquely through $A_{\widetilde{A}_B}'$. Since $A\to A_{\widetilde{A}_B}'$ is also birational then $A\to \widetilde{A}_{A_B '}$
factorizes uniquely through $A_{\widetilde{A}_B}'$ by Proposition \ref{PUbiratclos}. Conversely, the extension $A\to \widetilde{A}_{A_B '}$ is birational so $A\to \widetilde{A}_B$ factorizes uniquely through 
$\widetilde{A}_{A_B '}$. Since $A\to \widetilde{A}_{A_B '}$ is also integral then $A\to A_{\widetilde{A}_B}'$ factorizes uniquely through $\widetilde{A}_{A_B '}$.
\end{proof}

\begin{defn} \label{defintbiratclos}
Let $A\to B$ be an extension of domains. We simply denote by $\widetilde{A}_B'$ the ring $A_{\widetilde{A}_B}'=\widetilde{A}_{A_B '}$ and we call it the integral and birational closure of $A$ in $B$.
\end{defn}

From above results we easily get an universal property for the integral and birational closure.
\begin{prop} \label{PUintbiratclos}
Let $A\to B$ be an extension of domains. Then, $\widetilde{A}_B'$ is intermediate between $A$ and $B$ with $A\to \widetilde{A}_B'$ integral and birational and moreover if $C$ is an intermediate ring between $A$ and $B$ and $A\to C$ is integral and birational then $C\to B$ factorizes uniquely through $\widetilde{A}_B'$.
\end{prop}

\subsection{Resolution of the Central Seminormalization Existence Problem for rings}

We prove the existence of the central seminormalisation of a ring in another one in the following theorem \ref{pb1}.
\begin{thm} 
  \label{pb1}
Let $A\to B$ be an extension of domains. The
central seminormalization $A^{s_c,*}_B$ of $A$ in $B$ exists and
moreover we have: 
\begin{enumerate}
\item[1)] If $A$ is a real domain then $$A^{s_c,*}_B=A^{s_c}_{\widetilde{A}_B'}=\{b\in \widetilde{A}_B'|\,\,\forall\p\in\RCent
A,\,\,b_{\p}\in A_{\p}+\JRadCe((\widetilde{A}_B')_{\p})\}$$
\item[2)] If $A$ is not a real domain then $$A^{s_c,*}_B=A^\prime_B=A^{s_c}_{A_B'}$$
\end{enumerate}
\end{thm}

\begin{proof}
Assume $A$ is not a real domain. It follows from Remark \ref{triviallyempty} that $A\to A^\prime_B$ is trivially $s_c$-subintegral. Since a $s_c$-subintegral extension is integral then we get 2).

Let us prove 1). We assume $A$ is a real domain. Let $C$ be an intermediate domain between $A$ and $B$ such that $A\to
C$ is $s_c$-subintegral. By Remark \ref{remscbirat}, it follows that $A\to C$ is integral and birational and thus from Proposition \ref{PUintbiratclos} we get
$C\subset \widetilde{A}_B'$. By Corollary \ref{inclnormglue}, we get
$C\subset A^{s_c}_{\widetilde{A}_B'}$.
  
To end the proof it is sufficient to prove $A\to
A^{s_c}_{\widetilde{A}_B'}$ is $s_c$-subintegral.
We know that $A^{s_c}_{\widetilde{A}_B'}$ satisfies the properties (1) and
(2) of Theorem \ref{PUgluingAB} for the extension $A\to \widetilde{A}_B'$. It means that the map $\RCent
A^{s_c}_{\widetilde{A}_B'}\to\RCent A$ is injective and equiresidual by restriction to the
image of $\RCent (\widetilde{A}_B')\to\RCent A^{s_c}_{\widetilde{A}_B'}$. Since $A\to A^{s_c}_{\widetilde{A}_B'}$ and
$A^{s_c}_{\widetilde{A}_B'}\to \widetilde{A}_B'$ are integral and birational then the maps $\RCent
A^{s_c}_{\widetilde{A}_B'}\to\RCent A$ and $\RCent (\widetilde{A}_B')\to \RCent A^{s_c}_{\widetilde{A}_B'}$ are
surjective (Proposition \ref{lying-over}) and it gives the desired conclusion.
\end{proof}

For integral and birational extensions, we do not have to distinguish the empty case and we get:
\begin{cor} 
  \label{pb1cor}
Let $A\to B$ be an integral and birational extension of domains. The
central seminormalization of $A$ in $B$ and
the central gluing of $B$ over $A$ coincide i.e $$A^{s_c,*}_B=A^{s_c}_{B}$$
\end{cor}

\begin{rem} In the special case $B=A'$, Corollary \ref{pb1cor} gives \cite[Prop. 2.23]{FMQ-futur2}.
\end{rem}

\begin{cor}
  \label{corpb1}
Let $A\to B$ be an extension of domains. The following
properties are equivalent:
\begin{enumerate}
\item[1)] $A$ is $s_c$-normal in $B$.
\item[2)] $\begin{cases} A=A^{s_c}_{\widetilde{A}_B'}\,\,{\rm if}\,A\,{\rm is\,\,a\,\,real\,\,domain},\\
A\,\,{\rm is\,\,integrally\,\,closed\,\,in}\,\,B\,\,{\rm else}.\end{cases}$
\end{enumerate}
\end{cor}

\begin{ex}
Consider the extension $A=\R[x,y]/(y^2-x^3(x-1)^2(2-x))\to \R[x,z,u,v]/(z^2-x(2-x), u^4+z^2+1)=B$ such that $x\mapsto x$, $y\mapsto zx(x-1)$.
We may decompose $A\to B$ in the following way:\\
1. $A\to \R[x,Y]/(Y^2-x(x-1)^2(2-x))=C$ such that $x\mapsto x$ and $y\mapsto Yx$. This extension is clearly $s_c$-subintegral.\\
2. $C\to  \R[x,z]/(z^2-x(2-x))=D$ such that $x\mapsto x$ and $Y\mapsto z(x-1)$. This extension is integral and birational. Remark that $D$ is the integral and birational closure of $A$ and that $C$ is $s_c$-normal in $D$.\\
3. $D\to \R[x,z,u]/(z^2-x(2-x), u^4+z^2+1)=E$. This extension is integral but not birational.\\
4. $E\to B=E[v]$. \\
We get here that $A_B'=E$, $\widetilde{A}_B'=D$ and $A^{s_c,*}_B=C$. Since $\RCent E=\emptyset$ then $A^{s_c}_E=E$ and thus
$$A^{s_c,*}_B=A^{s_c}_{\widetilde{A}_B'}\not=A^{s_c}_{A_B'}=A_B'$$
and it proves that in 1) of Theorem \ref{pb1} it is necessary to consider the integral and birational closure and not only the integral closure before doing the central gluing.
\end{ex}

\section{Traverso's type structural decomposition theorem}

We have already proved the existence of a central seminormalization of a ring in another one. If we take an explicit geometric example i.e a finite extension of coordinate rings of two irreducible affine algebraic varieties over $\R$, due to the fact that when we do the central gluing then we glue together infinitely many ideals, it is in general not easy to compute the central seminormalization. In the main result of the paper, we prove that, under reasonable hypotheses, we can obtain the central seminormalization by only a finite number of central gluings and a birational gluing. More precisely, consider a finite extension $A\to B$ of noetherian domains with $A$ a real domain. We want
to show that the $s_c$-normalization $A^{s_c,*}_B$ of $A$ in $B$ can be
obtained from $B$ by the standard gluing over the null ideal of $A$ followed by a finite sequence of elementary central gluings over a
finite set of central prime ideals of $A$ like Traverso's decomposition
theorem for classical seminormal extensions \cite{T}. This result allows to prove in the next section that the processes of central seminormalization and localization commute
together. In this section, we make strong use of the results developed in Section \ref{sectCalg}.

\subsection{Central seminormality and conductor}

We prove in this section that the $s_c$-normality of an
extension is strongly related to a property of the conductor.

Let $A\to B$ be an extension of rings. We recall that the conductor of
$A$ in $B$, denoted by $(A:B)$, is the set $\{a\in A|\,aB\subset
A\}$. It is an ideal in $A$ and also in $B$.

\begin{prop}
\label{normcond}
Let $A\to B$ be an extension of domains with $A$ a real domain.
If $A$ is $s_c$-normal in $B$ then $(A:\tilde{A}_B')$ is a central
ideal in $\widetilde{A}_B'$.
\end{prop}

\begin{proof}
Assume $A$ is $s_c$-normal in $B$. From Proposition \ref{intermednorm}, we know that $A$ is $s_c$-normal in $\widetilde{A}_B'$. For the rest of the proof, we replace $B$ by $\widetilde{A}_B'$.
Let $I=(A:B)$, it is the
biggest ideal in $B$ contained in $A$. By Corollary \ref{corradcent}, we have to show that $\RadCe
I\subset A$ where $\RadCe I$ is the central radical of $I$ in $B$.
Let $b\in B$ such that $b\in \RadCe I$ and let $\p\in\RCent A$.\\
$\bullet$ Assume $I\subset \p$. Since $b\in \RadCe I$
then $b\in\bigcap_{\q\in\RCent B,\,\q\cap A=\p}\q.$ Thus $b\in
\JRadCe (B_{\p})$.\\
$\bullet$ Assume $I\not\subset \p$. Then $B_{\p}=A_{\p}$.\\
We have proved that $b\in A_B^{s_c}$. The proof is done since
$A=A_B^{s_c}$ (Corollary \ref{pb1cor}).
\end{proof}

\begin{cor}
  \label{corcond}
Let $A\to B$ be an extension of domains with $A$ a real domain.
If $A$ is $s_c$-normal in $B$ then $(A:\tilde{A}_B')$ is a central
  ideal in $A$.
\end{cor}

\begin{proof}
Since the contraction of a central ideal remains central (Proposition \ref{contractcentral}) then the proof follows from Proposition \ref{normcond}.
\end{proof}

\begin{prop}
  \label{transitivity}
Let $A\to B\to C$ be a sequence of integral and birational extensions of domains.
If $A$ is $s_c$-normal in $B$ and $B$ is
$s_c$-normal in $C$ then $A$ is $s_c$-normal in $C$.
\end{prop}

\begin{proof}
Let $c\in A_C^{s_c}$ and let
$\q\in\RCent B$. Let $\p=\q\cap A\in\RCent A$. We have
$c=\alpha+\beta$ with $\alpha\in A_{\p}$ and $\beta\in\JRadCe
(C_{\p})$. Since the central prime ideals of $C$ lying over $\q$ ly over
$\p$ we get $\JRadCe (C_{\p})\subset \JRadCe(C_{\q})$ (the inclusion
is seen in $\K(A)=\K(B)=\K(C)$). Since $A_{\p}\subset B_{\q}\subset \K(A)$ then
$c\in B_{\q}+\JRadCe(C_{\q})$. It follows that $b\in B_C^{s_c}=B$
(Corollary \ref{pb1cor}) and
thus $A_C^{s_c}\subset B$. Since $A\to A_C^{s_c}$ is
$s_c$-subintegral then we get $A_C^{s_c}\subset A_B^{s_c}$. Since $A$ is $s_c$-normal in $B$ then we get
$A=A_C^{s_c}$ i.e $A$ is is $s_c$-normal in $C$.
\end{proof}

\begin{prop}
  \label{notnormintermed}
Let $A\to B$ be an extension of 
domains. Let $C$ be an intermediate ring between $A$ and $A^{s_c,*}_B$. If
$C\not=A^{s_c,*}_B$ then $C$ is not $s_c$-normal in $B$.
\end{prop}

\begin{proof}
Assume $C\not=A^{s_c,*}_B$ and $C$ is $s_c$-normal in $B$. By
Proposition \ref{propcentsubint} we get that $C\to A^{s_c,*}_B$ is
$s_c$-subintegral contradicting the $s_c$-normality of $C$ in $B$.
\end{proof}

\subsection{Elementary central gluings}

We will adapt and revisit the definition of elementary gluing developed by Traverso \cite{T} in the central case. 

\subsubsection{Non-trivial elementary central gluings}

Throughout this section we consider the following situation:
Let $A\to B$ be an integral extension of domains with $A$ a real domain and let $\p\in\RCent
A$. We assume the set of central prime ideals of $B$ lying over $\p$ is
finite and non-empty (it follows from Propositions \ref{existcentralideal} and \ref{contractcentral} that $B$ is also a real domain), we denote by $\q_1,\ldots,\q_t$ these prime
ideals. Remark that by Proposition \ref{lying-over} the previous
condition is met if $A\to B$ is finite and birational. We denote by $\gamma_i:k(\p)\to k(\q_i)$ $i=1,\ldots,t$ the canonical
extensions of the residue fields. Let $Q=\prod_{i=1}^t k(\q_i)$ and
$\gamma=\prod_{i=1}^t\gamma_i$ be the injection $k(\p)\to Q$. Remark that $\gamma$ identifies $k(\p)$ with a subset of the diagonal of $Q$.

\begin{defn}
  \label{defgluing}
We define the ``central gluing of $B$ over
$\p$'', denoted by $A^{s_c,\p}_B$ defined as the fibre product $B\times_Q k(\p)$ i.e
the domain defined by the following pull-back diagram of
commutative rings \[
\begin{array}{ccccc}
& A^{s_c,\p}_B & \stackrel{i}\to & B & \\
&h\downarrow &&\downarrow g& \\
& k(\p) &\stackrel{\gamma}\to & Q& \\
\end{array}
\]\\
where $g$ is the composite $B\to \prod_{i=1}^t B/\q_i\to Q$. A central
gluing over a central prime ideal of this type is called a ``non-trivial elementary
central gluing''.
\end{defn}

\begin{rem} The term "non-trivial" in the previous definition corresponds to the fact that the set of central primes of $B$ lying over $\p$ is non-empty and thus we really glue something.
\end{rem}

\begin{rem}
Back to the diagram of Definition \ref{defgluing}. The map $i$
is an injection and $i(A^{s_c,\p}_B)$ is the set of elements $b$ in $B$ such that
$g(b)\in \gamma(k(\p))$. In particular we have that $i(A^{s_c,\p}_B)$ contains $A$. We identify $A^{s_c,\p}_B$ with $i(A^{s_c,\p}_B)$.
From the above commutative diagram, the ring $A^{s_c,\p}_B$ is determined by $$A^{s_c,\p}_B=\{b\in B|\,\forall
i\in\{1,\ldots,t\}\,\,b(\q_i)\in k(\p)\,\,(*)\,{\rm and}\,\forall
(i,j)\in\{1,\ldots,t\}^2\,\,b(\q_i)=b(\q_j)\,\,(**)\,\}.$$
\end{rem}

\begin{prop}
  \label{propertygluing}
We set $\q=\cap_{i=1}^t\q_i\cap A^{s_c,\p}_B$. We have:
\begin{enumerate}
\item[1)] $\cap_{i=1}^t\q_i\subset A^{s_c,\p}_B$ so
  $\q=\cap_{i=1}^t\q_i$ and $\q \subset (A^{s_c,\p}_B:B)$.
\item[2)] $\forall i\in\{1,\ldots,t\}$, $\q_i\cap A^{s_c,\p}_B=\q$ and thus $\q$
  is a central prime ideal of $A^{s_c,\p}_B$ lying over $\p$.
\item[3)] The extension $k(\p)\to k(\q)$ is an
  isomorphism.
\end{enumerate}
\end{prop}

\begin{proof}
Let $b\in \cap_{i=1}^t\q_i$. We have $\forall i\in\{1,\ldots,t\}$,
$b(\q_i)=0\in k(\p)$ so $b\in A^{s_c,\p}_B$ and we get 1).

Let $b\in\q_1\cap A^{s_c,\p}_B$. We have $b(\q_1)=0$ and $h(b)\in
k(\p)$. Thus $\forall i\in\{1,\ldots,t\}$ we get $(\gamma_i\circ h)
(b)=(\gamma_1\circ h)(b)=b(\q_i)=0$. It follows that $\gamma\circ
h(b)=(0,\ldots,0)\in Q$. It means $g(b)=(0,\ldots,0)$ and thus $b\in
\q=\ker g$. Hence $\q_1\cap A^{s_c,\p}_B\subset \q$ and thus $\q_1\cap
A^{s_c,\p}_B=\q$. Since $\q=\q_1\cap A^{s_c,\p}_B$ and $\q_1$ is central then $\q$ is also
central (Proposition \ref{contractcentral}). It is clear that $\q$
is an ideal in $A^{s_c,\p}_B$ and $B$ and thus $\q \subset (A^{s_c,\p}_B:B)$. Therefore the assertion 2) is true.

Since $\q=\ker (g\circ i)$ then it follows from the above commutative
diagram that $(A^{s_c,\p}_B/\q)\subset k(\p)$ and thus $k(\q)\subset k(\p)$. Since
by 2) $\q$
  is a prime ideal of $A^{s_c,\p}_B$ lying over $\p$ then we get
  $k(\p)\subset k(\q)$, and it gives 3).
\end{proof}

A non-trivial elementary central gluing satisfies the following universal property.
\begin{prop}
  \label{PUgluing}
The ring $ A^{s_c,\p}_B$ is
the biggest intermediate ring $C$ between $A$ and $B$ satisfying:
\begin{enumerate}
\item $\forall(i,j)\in\{1,\ldots,t\}^2$, $\q_i\cap C=\q_j\cap C$.
\item $\forall i\in\{1,\ldots,t\}$, $k(\p)\to k(\q_i\cap C)$ is an
  isomorphism.
\end{enumerate}
\end{prop}

\begin{proof}
The ring $A^{s_c,\p}_B$ satisfies (1) and (2) by Proposition \ref{propertygluing}. It is
clear that a subring of $B$ containing
$A$ satisfying (1) and (2) satisfies ($*$) and ($**$) and thus is
contained in $A^{s_c,\p}_B$.
\end{proof}

\begin{prop}
  \label{defequivgluing}
We have $$A^{s_c,\p}_B=\{b\in
B|\,b\in A_{\p}+\JRadCe(B_{\p})\}.$$
\end{prop}

\begin{proof}
Let $D$ denote the ring $\{b\in
B|\,b\in A_{\p}+\JRadCe(B_{\p})\}$.

From the proof of Theorem \ref{PUgluingAB}  we see that $D$ is the biggest subring of $B$
satisfying properties (1) and (2) of Proposition \ref{PUgluing},
therefore we get the proof.
\end{proof}

In case $A\to B$ is birational with $A$ a real domain then any elementary central gluing is non-trivial and we can strengthen the universal
property of (non-trivial) elementary central gluings.
\begin{prop}
  \label{PUgluingbirat}
Under the above notation and hypotheses, we assume in addition $A\to B$ is birational. The ring $A^{s_c,\p}_B$
defined above is the biggest element among the subrings $C$ of $B$ containing
$A$ satisfying:
\begin{enumerate}
\item[a)] There exists a unique ideal $\q'$ in $\RCent C$ lying over $\p$.
\item[b)] $k(\p)\to k(\q')$ is an
  isomorphism.
\end{enumerate}
\end{prop}

\begin{proof}
We first prove that $A^{s_c,\p}_B$ satisfies the properties a) and b) of the
proposition. Let $\q'\in\RCent A^{s_c,\p}_B$ lying over $\p$. By Proposition
\ref{lying-over} there exists a central prime ideal $\q''$ of $B$
lying over $\q'$. Since $\q'$ is lying over $\p$ then
$\q''$ must be one of the $\q_i$ and thus $\q'=\q_i\cap A^{s_c,\p}_B$. The properties a)
and b) comes from (1) and (2) of Proposition \ref{PUgluing}.

Let $C$ be an intermediate ring between $A$ and $B$ satisfying a) and
b). Let $\q'$ be the unique central prime ideal of $C$ lying over
$\p$. As above arguments there exists one of the $\q_i$ lying over
$\q'$. By unicity we must have $\forall(i,j)\in\{1,\ldots,t\}^2$,
$\q_i\cap C=\q_j\cap C=\q'$ and thus $C$ satisfies (1) of Proposition
\ref{PUgluing}. By b) we get that $C$ satisfies (2) of Proposition
\ref{PUgluing} and thus $C\subset A^{s_c,\p}_B$.
\end{proof}

We give some properties of non-trivial elementary central gluings. We start with a $s_c$-normality property.
\begin{prop}
  \label{prop5}
We have
$A^{s_c,\p}_B=(A^{s_c,\p}_B)^{s_c}_B$ and thus $A^{s_c,\p}_B$ is
$s_c$-normal in $B$.
\end{prop}

\begin{proof}
By Theorem \ref{PUgluingAB} it is clear that $(A^{s_c,\p}_B)^{s_c}_B$
satisfies the properties (1) and (2) of Proposition \ref{PUgluing} and
thus $A^{s_c,\p}_B=(A^{s_c,\p}_B)^{s_c}_B$. By Corollary
\ref{inclnormglue}, it follows that $A^{s_c,\p}_B$ is
$s_c$-normal in $B$.
\end{proof}

We prove that the operations of localization and non-trivial elementary central gluings commute together.
\begin{prop}
  \label{prop7}
We assume $S$ is a
multiplicative closed subset of $A$. We have:
\begin{enumerate}
\item[1)] If $S\cap \p\not=\emptyset$ then $S^{-1}(A^{s_c,\p}_B)=S^{-1}B$.
\item[2)] If $S\cap\, \p=\emptyset$ then $S^{-1}(A^{s_c,\p}_B)$ is the non-trivial elementary
  central gluing of $S^{-1}B$ over $S^{-1}\p$ i.e
  $$S^{-1}A^{s_c,\p}_B=(S^{-1}A)^{s_c,S^{-1}\p}_{S^{-1}B}.$$
\end{enumerate}
In particular,
$$S^{-1}(A^{s_c,\p}_B)=(S^{-1}(A^{s_c,\p}_B))^{s_c}_{S^{-1}B}$$ and thus
$S^{-1}(A^{s_c,\p}_B)$ is $s_c$-normal in $S^{-1}B$.
\end{prop}

\begin{proof}
Let $\q= \cap_{i=1}^t\q_i$.
If $S\cap \p\not=\emptyset$ then $S^{-1}\q=S^{-1}(A^{s_c,\p}_B)$. The
conductor commutes with localization so
$S^{-1}(A^{s_c,\p}_B:B)=(S^{-1}(A^{s_c,\p}_B):S^{-1}B)$ contains
$S^{-1}\q=S^{-1}(A^{s_c,\p}_B)$ (Proposition \ref{propertygluing}) and thus 
$S^{-1}(A^{s_c,\p}_B)=S^{-1}B$.

Assume $S\cap\, \p=\emptyset$. Since $S^{-1}\p$, $S^{-1}\q$,
are central prime ideals, since the
$S^{-1}\q_i $, $i=1,\ldots,t$, are the central prime ideals of
$S^{-1}B$ lying over $S^{-1}\p$, since
$S^{-1}\q=\cap_{i=1}^t(S^{-1}\q_i)$, localization commutes with
quotient, and since $k(\p)=k(S^{-1}\p)$,
$k(\q)=k(S^{-1}\q)$, $k(\q_i)=k(S^{-1}\q_i)$ for $i=1,\ldots,t$ then
$$(S^{-1}A)^{s_c,S^{-1}\p}_{S^{-1}B}=S^{-1}B\times_{S^{-1}Q} k(S^{-1}\p)=S^{-1}B\times_Q k(\p)=S^{-1}(B\times_{Q} k(\p))=S^{-1}A^{s_c,\p}_B.$$
It shows that the following diagram \[
\begin{array}{ccccc}
& A^{s_c,\p}_B & \stackrel{i}\to & B & \\
&h\downarrow &&\downarrow g& \\
& k(\p) &\stackrel{\gamma}\to & Q& \\
\end{array}
\]\\ commutes with localization by $S$. The end of
the proof comes from Proposition \ref{prop5}.
\end{proof}

\begin{prop}
  \label{prop8}
Let $\p'\in\Sp A$ such that
$\p\not\subset \p'$. The prime ideals of $A^{s_c,\p}_B$ lying over
$\p'$ are in bijection with the prime ideals of $B$ lying over $\p'$
and moreover they have the same nature: real, non-real, central,
non-central.
\end{prop}

\begin{proof}
Let $\q= \cap_{i=1}^t\q_i$. By 1) of Proposition \ref{propertygluing},
we have $\q\subset (A^{s_c,\p}_B:B)$ thus $\p\subset
(A^{s_c,\p}_B:B)\cap A$ and therefore $(A^{s_c,\p}_B:B)\cap
A\not\subset \p'$. Thus $(A^{s_c,\p}_B)_{\p'}=B_{\p'}$ and the proof
is done.
\end{proof}

\subsubsection{Generalized elementary central gluings, the birational gluing and examples}

We generalize the concept of elementary
central gluing.
\begin{defn}
  \label{defgluing2}
Let $A\to B$ be a finite extension of domains with $A$ a real domain. Let $\p\in\RCent
A$. We define the ``central gluing of $B$ over
$\p$'', denoted by $A^{s_c,\p}_B$ defined as:
\begin{enumerate}
  \item[$\bullet$] If the set of central prime ideals of $B$ lying over
    $\p$ is non-empty then $A^{s_c,\p}_B$ is the ring defined in
    Definition \ref{defgluing} and we say that it is a non-trivial
    elementary central gluing.
  \item[$\bullet$] If not then $A^{s_c,\p}_B:=B$ and we say that it is
    a trivial elementary central gluing.
\end{enumerate}
\end{defn}

We can easily generalize Proposition \ref{defequivgluing} and show
that central gluings over rings can be written in terms of elementary
central gluings.
\begin{prop}
  \label{defequivgluing2}
Let $A\to B$ be a finite extension of domains with $A$ a real domain.
\begin{enumerate}
\item[1)] Let $\p\in\RCent
A$. We have $$A^{s_c,\p}_B=\{b\in
B|\,b\in A_{\p}+\JRadCe(B_{\p})\}.$$
\item[2)] The central gluing $A^{s_c}_B$ of $B$ over $A$ can be seen
  as simultaneous elementary central gluings of $B$ over all the central prime ideals of
  $A$. Namely, we have $$A^{s_c}_B=\bigcap_{\p\in\RCent A}
  A^{s_c,\p}_B.$$
\end{enumerate}
\end{prop}

The following property will be useful in the next section.
\begin{prop}
  \label{scsub+elementary}
Let $A\to C\to B$ be a sequence of two finite extensions of 
domains such that $A$ is a real domain and $A\to C$ is $s_c$-subintegral. Let $\p\in\RCent A$
and let $\p'\in\RCent C$ be the unique central prime ideal lying over
$\p$. We have $$A^{s_c,\p}_B=C^{s_c,\p'}_B.$$
\end{prop}

\begin{proof}
If the set of central prime ideals of $B$ lying over
    $\p$ is empty then the set of central prime ideals of $B$ lying over
    $\p'$ is empty and $A^{s_c,\p}_B=C^{s_c,\p'}_B=B$.

Assume the set of central prime ideals of $B$ lying over
    $\p$ is non-empty. Let $\q$ be one of these ideals. Since $\p'$ is
    the unique central ideal of $C$ lying over $\p$ and since $\q\cap
    C$ is central then $\q$ lies over $\p'$. We have proved that
    the central prime ideals of $B$ lying over $\p$ or $\p'$ are the
    same, since $k(\p)=k(\p')$, it follows from Definition
    \ref{defgluing} that $C^{s_c,\p'}_B=B\times_{Q} k(\p')=B\times_Q k(\p)=A^{s_c,\p}_B.$
\end{proof}

Elementary central gluings are not sufficient to get a decomposition theorem due to the presence of the birational closure in Theorem \ref{pb1}.
Let $A\to B$ be an integral extension of domains and let $\p\in\Sp A$. The elementary Traverso's gluing of $B$ over $\p$ can be defined similarly
as in Definition \ref{defgluing} but here we consider all the prime ideals of $B$ lying over $\p$ and not only the central ones.  Following the proof of Proposition \ref{defequivgluing},
this gluing is 
$$\{b\in
B|\,b\in A_{\p}+\JRad(B_{\p})\}$$ 
From above and Definition \ref{defbiratclos}, we see that the birational closure is an elementary Traverso's gluing:
\begin{prop}\label{biratclos=gluingnull}
Let $A\to B$ be an integral extension of domains. Then, the birational closure $\widetilde{A}_B$ of $A$ in $B$ is the elementary Traverso's gluing of $B$ over the null ideal of $A$.
\end{prop}

In the sequel, the "birational closure" will be also called the "birational gluing".
The birational gluing is sometimes an elementary central gluing:
\begin{prop}
  \label{gluingnull}
Let $A\to B$ be an integral extension of real domains. The central
gluing $A^{s_c,(0)}_B$ of $B$ over the null ideal of $A$ is the birational closure $\widetilde{A}_B$ of $A$ in $B$.
\end{prop}

\begin{proof}
By Proposition \ref{defrealdomain}, the null ideal of $B$ is the unique central prime ideal of
$B$ lying over the null ideal of $A$. The commutative diagram of Definition \ref{defgluing} becomes in this situation
 \[
\begin{array}{ccccc}
& A^{s_c,(0)}_B & \stackrel{i}\to & B & \\
&h\downarrow &&\downarrow g& \\
& \K(A) &\stackrel{\gamma}\to & \K(B)& \\
\end{array}
\]\\
and thus $A^{s_c,(0)}_B=B\times_{\K(B)}\K(A)=\widetilde{A}_B$.
\end{proof}

\begin{ex}
  \label{parabole}
Consider the following finite extension of real domains $A=\R[x]\to
B=\R[x,y]/(y^2-x)$ sending $x$ to itself. Then $A$ is the central
gluing of $B$ over the null ideal of $A$ i.e $A=B\times_{\K(B)} \K(A)$
(Proposition \ref{gluingnull}). Indeed, if $f\in B$ 
then we may write $f=p+yq$ with $p,q\in\R[x]$ and if moreover $f\in
\K(A)$ then $q=0$.
\end{ex}

\begin{ex}
\label{parabole2}
Consider the following finite extension of domains $A=\R[x]\to
B=\R[x,y]/(y^2+x^2)$ sending $x$ to itself. Remark that the null ideal of $B$ is not a central ideal since $-1$ is a square in $\K(B)$ (see Proposition \ref{defrealdomain}).
As in the previous example, we can prove that $A$ is the birational gluing of $B$ over $A$ but here the birational gluing is not an elementary central gluing.
\end{ex}

\begin{ex}
  \label{Whithney}
Let $V$ be the Whitney umbrella i.e the affine algebraic surface over
$\R$ with coordinate ring $\R[V]=\R[x,y,z]/(y^2-zx^2)$ and let $V'$ be
its normalization. The coordinate ring of $V'$ is
$\R[V']=\R[x,Y,z]/(Y^2-z)$ and consider the finite birational
extension $\R[V]\to\R[V']$ given by sending $x$ to $x$, $y$ to $Yx$,
$z$ to $z$. We claim $\R[V]$ is equal to the central gluing of $\R[V']$
over the central prime ideal $\p=(x,y)$ of $\R[V]$. There is a unique
prime ideal of $\R[V']$ lying over $\p$, that we denote by $\q$, and
$\q=(x,Y^2-z)$ is also a central ideal. We have $k(\p)=\R(z)$ and
$k(\q)=\R(z)[Y]/(Y^2-z)$. Let $f\in \R[V]^{s_c,\p}_{\R[V']}=\R[V']\times_{k(\q)} k(\p)$, we may write
$f=g+Yv$ with $g,v\in\R[x,z]$. From the following commutative diagram
\[
\begin{array}{ccccc}
& \R[V]^{s_c,\p}_{\R[V']} & \stackrel{i}\to & \R[V'] & \\
&h\downarrow &&\downarrow g& \\
& \R(z) &\stackrel{\gamma}\to & \R(z)[Y]/(Y^2-z)& \\
\end{array}
\]\\
we see that $x$ must divide $v$ and thus $v=xs$ with $s\in\R[x,z]$. It
follows that $f=g+Yxs=g+ys\in\R[V]$ and it proves the claim.
\end{ex}

\subsection{Structural decomposition theorem}

As announced, we show that if a noetherian domain $A$ is centrally seminormal in a domain $B$ which is a finite $A$-module
then $A$ can be obtained from $B$ by the birational gluing followed by a finite number of successive
non-trivial elementary central gluings.

\begin{thm}
  \label{structuralthm}
Let $A\to B$ be a finite extension of domains and assume $A$ is a
noetherian ring. If $A$ is $s_c$-normal in $B$ then there is a finite
sequence $(B_i)_{i=0,\ldots,n}$ of real domains such that:
\begin{enumerate}
\item[1)] $A=B_n\subset \cdots\subset B_1\subset B_0=B$.
\item[2)] $B_1$ is the birational gluing of $B$ over $A$.
\item[3)] for $i\geq 1$, $B_{i+1}$ is the central gluing of $B_i$ over a central
  prime ideal of $A$.
\end{enumerate}
\end{thm}

\begin{proof}
We assume $A$ is $s_c$-normal in $B$. If $A$ is not a real domain then $A=B$ (Corollary \ref{corpb1}) and there is nothing to do. In the sequel of the proof we assume $A$ is a real domain.
  
First remark that $B$ is also a noetherian ring since it is a
noetherian $A$-module. Indeed a finite module over a noetherian ring
is a noetherian module.

Since $A$ is $s_c$-normal in $B$ then $A$ is
$s_c$-normal in $\widetilde{A}_B$ (Proposition \ref{intermednorm}). Since
$\widetilde{A}_B$ is the birational gluing of $B$ over $A$
(Proposition \ref{biratclos=gluingnull}) and since 
$A\to \widetilde{A}_B$ is finite (every submodule of a noetherian module is
finite) then we may assume $A\to B$ is birational in the rest of the proof. 

Assume we have already builded the sequence from $B_0$ to $B_i$ and
moreover that $B_i\not= A$. We denote $(A:B_i)$ simply by $I$. By
Proposition \ref{intermednorm} then $A$ is $s_c$-normal in $B_i$. By
Proposition \ref{normcond} and 
Corollary \ref{corcond} then $I$ is central in $B_i$ and also in
$A$. By Propositions \ref{centralnoeth} and \ref{propradcent}, the
minimal prime ideals of $A$ containing $I$ (in finite number by
noetherianity) are all central ideals and their intersection is
$I$. Let $\p$ be one of these minimal prime ideals. Set
$B_{i+1}=A^{s_c,\p}_{B_i}$ the central gluing of $B_i$ over $\p$ and
set $J=(A:B_{i+1})$.\\
We claim that $J\not\subset \p$:\\
Suppose $J\subset \p$. We have $I\subset J$ and
since $\p$ is a minimal prime ideal of $A$ containing $I$ then $\p$ is
also a minimal prime ideal of $A$ containg $J$. Since $A$ is
$s_c$-normal in $B$ then $A$ is $s_c$-normal in $B_{i+1}$ (Proposition
\ref{intermednorm}) and thus $J$ is a central ideal (Proposition \ref{normcond} and 
Corollary \ref{corcond}). We denote by $\q$ the unique central prime
ideal of $B_{i+1}$ lying over $\p$ given by Proposition
\ref{PUgluingbirat}. We localize in $\p$, we have
$J_{\p}=(A_{\p}:(B_{i+1})_{\p})=\p A_{\p}$ since $\p$ is a primary
component of $J$. Since $J_{\p}$ is a central ideal in
$(B_{i+1})_{\p}$ then it is the
intersection of the central prime ideals of $(B_{i+1})_{\p}$
containing it (Proposition \ref{propradcent}) so
\begin{equation} \label{equ1} J_{\p}=\p A_{\p}=\q(B_{i+1})_{\p}
\end{equation}
By Proposition
\ref{PUgluingbirat} we have $k(\p)=k(\q)$ and thus 
\begin{equation} \label{equ2} (A_{\p}/\p A_{\p})=((B_{i+1})_{\p}/\q (B_{i+1})_{\p})
\end{equation}
Let $b\in (B_{i+1})_{\p}$. By (\ref{equ2}) (we may also use
Proposition \ref{defequivgluing2} and 2) of Proposition \ref{propertygluing}) we may write $b=\alpha+\beta$ with
$\alpha\in A_{\p}$ and $\beta\in \q (B_{i+1})_{\p}$. By (\ref{equ1}),
we get $\beta\in\p A_{\p}$ and thus $b\in A_{\p}$. We have proved that
$A_{\p}=(B_{i+1})_{\p}$, this is impossible (since $J\subset \p$ by
hypothesis) and we get the claim.\\

We have $I\subset J$, $I\subset \p$ and $J\not\subset\p$. Therefore
$I\not= J$ and we may build a strictly ascending sequence of ideals as
soon as $B_i\not= A$. By noetherianity of $A$, we get the proof of the
theorem.
\end{proof}

\begin{cor}
  \label{corstructuralthm1}
Let $A\to B$ be a finite extension of real domains and assume $A$ is a
noetherian ring. If $A$ is $s_c$-normal in $B$ then $A$ can be obtained from $B$ by a finite number of successive
elementary central gluings over central prime ideals of $A$.
\end{cor}

\begin{proof}
Since here $A$ and $B$ are real domains then the birational gluing of $B$ over $A$ is an elementary central gluing (Proposition \ref{gluingnull}) and thus the proof follows from Theorem
\ref{structuralthm}.
\end{proof}

From Theorem
\ref{structuralthm} we get a structural decomposition theorem for the central seminormalization of $A$ in $B$ with gluings over central prime ideals of $A^{s_c,*}_B$ and the birational gluing.
We prove now a structural decomposition theorem for the central seminormalization of $A$ in $B$ using only gluings over central prime ideals of $A$ and the birational gluing.
\begin{thm}
  \label{corstructuralthmbis}
Let $A\to B$ be a finite extension of domains and assume $A$ is a 
noetherian ring. The central seminormalization $A^{s_c,*}_B$ of $A$ in
$B$ is $B$ (if $A$ is not a real domain) or can be obtained from $B$ by the birational gluing over $A$ followed by a finite number of successive
elementary central gluings over central prime ideals of $A$.
\end{thm}

\begin{proof}
If $A$ is not a real domain then $A^{s_c,*}_B=B$ by Theorem \ref{pb1} and there is nothing to do in that case. We assume $A$ is a real domain in the sequel of the proof.

The extension $A\to A^{s_c,*}_B$ is finite since every submodule of a noetherian module is
finite. It follows that $A^{s_c,*}_B$ is a noetherian ring and it is also a real domain (see Remark \ref{remscbirat}).  By Theorem \ref{pb1}, $A^{s_c,*}_B$ is $s_c$-normal in $B$.
From Theorem \ref{structuralthm}, $A^{s_c,*}_B$ can be obtained from $B$ by the birational gluing of $B$ over $A^{s_c,*}_B$ followed by a finite number of successive
elementary central gluings over central prime ideals of $A^{s_c,*}_B$.  Since $A\to A^{s_c,*}_B$ is birational, it follows from Proposition \ref{PUbiratclos} that the birational gluings 
of $B$ over $A$ and $A^{s_c,*}_B$ are the same. Since $A\to A^{s_c,*}_B$ is $s_c$-subintegral then it follows from Proposition \ref{scsub+elementary} that an elementary central gluing
(of an intermediate ring between $A^{s_c,*}_B$ and $B$) over a central prime ideal of $A^{s_c,*}_B$ is an elementary central gluing over a central prime ideal of $A$. The proof is done.
\end{proof}

\begin{cor}
  \label{corstructuralthmbis1}
Let $A\to B$ be a finite extension of real domains and assume $A$ is a 
noetherian ring. The central seminormalization $A^{s_c,*}_B$ of $A$ in
$B$ can be obtained from $B$ by a finite number of successive
elementary central gluings over central prime ideals of $A$.
\end{cor}

From Corollaries \ref{pb1cor} and \ref{corstructuralthmbis1} we get:
\begin{cor}
  \label{corstructuralthmbis2}
Let $A\to B$ be a finite and birational extension of real domains and assume $A$ is a
noetherian ring. The central gluing $A^{s_c}_B$ of $B$ over $A$ can be
obtained from $B$ by a finite number of successive
elementary central gluings over central prime ideals of $A$.
\end{cor}

We want to replace the word "successive" by "simultaneous" in the statement of Corollary \ref{corstructuralthmbis2}.
\begin{lem}
\label{lemsimult}
Let $A\to C\to B$ be a sequence of two integral and birational extensions of real domains. Let $\p\in\RCent A$. Then
$$A^{s_c,\p}_C=A^{s_c,\p}_B\cap C=C\times_B A^{s_c,\p}_B$$
\end{lem}

\begin{proof}
Since $C\to B$ is integral and birational then it follows from 4) of Proposition \ref{lying-over} that $\JRadCe B_{\p}\cap C_{\p}=\JRadCe C_{\p}$. From Proposition \ref{defequivgluing} it follows
that $$A^{s_c,\p}_C=\{c\in
C |\,b\in A_{\p}+\JRadCe(C_{\p})\}=\{c\in
C |\,b\in A_{\p}+(\JRadCe(B_{\p})\cap C)\}$$ $$=(\{b\in
B |\,b\in A_{\p}+\JRadCe(B_{\p})\})\cap C=A^{s_c,\p}_B\cap C=C\times_B A^{s_c,\p}_B$$
\end{proof}

\begin{prop}
\label{thmsimult}
Let $A\to B$ be a finite and birational extension of real domains and assume $A$ is a
noetherian ring. If $A$ is $s_c$-normal in $B$ and $A\not=B$ then there exist a finite number $\p_1,\ldots,\p_n$ of central prime ideals of $A$ such that
$A$ can be obtained by simultaneous elementary central gluings of $B$ over $\p_1,\ldots,\p_n$, namely
$$A=\bigcap_{i=1}^n
  A^{s_c,\p_i}_B$$
\end{prop}

\begin{proof}
By Corollary \ref{corstructuralthm1}, there are a finite
sequence $(B_i)_{i=0,\ldots,n}$ ($n>0$ since $A\not= B$) of real domains and a finite number $\p_1,\ldots,\p_n$ of central prime ideals of $A$ such that:
\begin{enumerate}
\item[1)] $A=B_n\subset \cdots\subset B_1\subset B_0=B$.
\item[2)] $B_{i+1}$ is the elementary central gluing of $B_i$ over $\p_{i+1}$ for $i=0,\ldots,n-1$.
\end{enumerate}
By successive applications of  Lemma \ref{lemsimult}, for $i=0,\ldots,n-1$ we get that $$B_{i+1}=\bigcap_{j=1}^{i+1}
  A^{s_c,\p_j}_B$$
\end{proof}

From Corollary \ref{corstructuralthmbis2} and Proposition \ref{thmsimult}, we get:
\begin{cor}
  \label{corstructuralthmsimult1}
Let $A\to B$ be a finite and birational extension of real domains and assume $A$ is a
noetherian ring. The central gluing $A^{s_c}_B$ of $B$ over $A$ is the intersection of a finite number of elementary central gluings of $B$ over central prime ideals of $A$.
\end{cor}

\begin{ex} \label{kollar}
Let $V$ be the Koll\'ar surface i.e the affine algebraic surface over
$\R$ with coordinate ring $\R[V]=\R[x,y,z]/(y^3-(1+z^2)x^3)$ and let $V'$ be
its normalization. The coordinate ring of $V'$ is
$\R[V']=\R[x,Y,z]/(Y^3-(1+z^2))$ and consider the finite birational
extension $\R[V]\to\R[V']$ given by sending $x$ to $x$, $y$ to $Yx$,
$z$ to $z$. Remark that $V$ and $V'$ are both central. Let
$\p=(x,y)\in\RCent \R[V]$, we have $k(\p)=\R(z)$. Let $\q$ be the
unique real (and central) ideal of $\R[V']$ lying over $\p$, we have
$k(\q)=\R(z)(^{3}\sqrt{1+z^2})$. Let $W$ be the affine algebraic surface over $\R$ with coordinate ring $\R[V][y^2/x]$. Since $y^2/x\in \SRR(V(\R))$ then it follows from Theorem \ref{scgeom} that
$\R[V]\to\R[W]$ is $s_c$-subintegral. 
To ilustate Theorem \ref{structuralthm}, we claim that $\R[W]$ can be obtained
from $\R[V']$ by a unique elementary central gluing, namely
$\R[W]=\R[V]^{s_c,\p}_{\R[V']}$. Since $\R[V]\to\R[W]$ is $s_c$-subintegral then $\R[W]\subset \R[V]^{s_c,*}=\R[V]^{s_c}_{\R[V']}$. By Proposition \ref{prop5} we get 
$\R[V]^{s_c,*}\subset (\R[V]^{s_c,\p}_{\R[V']})^{s_c,*}=\R[V]^{s_c,\p}_{\R[V']}$ and thus we have $$\R[W]\subset \R[V]^{s_c,\p}_{\R[V']}.$$
Let $f\in \R[V]^{s_c,\p}_{\R[V']}$, we may write
$f=g+Yh+Y^2t$ with $g,h,t\in\R[x,z]$. From the following commutative diagram
\[
\begin{array}{ccccc}
& \R[V]^{s_c,\p}_{\R[V']} & \stackrel{i}\to & \R[V'] & \\
&h\downarrow &&\downarrow g& \\
& \R(z) &\stackrel{\gamma}\to & \R(z)[Y]/(Y^3-(1+z^2))& \\
\end{array}
\]\\
we see that $x$ must divide $h$ and also $t$ and thus
$h=xs$, $t=x r$ with $s,r\in\R[x,z]$. It
follows that $f=g+Yxs+Y^2xr=g+ys+(y^2/x)r\in\R[W]$ and it proves the claim. Since $\R[V]\to\R[W]$ is $s_c$-subintegral, it follows from Proposition \ref{prop5} that 
$\R[W]$ is the $s_c$-normalization of $\R[V]$ i.e $$\R[W]=\R[V]^{s_c,*}.$$
\end{ex}

\begin{ex} \label{curve}
Let $n\in\N\setminus\{0\}$ and let $C$ be the affine plane algebraic
curve over $\R$ with coordinate ring
$\R[C]=\R[x,y]/(y^2-x\prod_{i=1}^n(x-i)^2)$. Let $C'$ be the
normalization of $C$, we have $\R[C']=\R[x,Y]/(Y^2-x)$ and the finite
birational extension $\R[C]\to\R[C']$ is given by sending $x$ to $x$
and $y$ to $Y\prod_{i=1}^n(x-i)$. The curve $C'$ is smooth and the
curve $C$ has only nodal and central singularities corresponding to the
maximal ideals $\m_i=(x-i,y)$ of $\R[C]$ for $i=1,\ldots,n$. We denote
by $\m_i'$ and $\m_i''$ the two distincts ideals of $\R[C']$ lying over
$\m_i$, we have $\m_i'=(x-i,Y-\sqrt{i})$ and
$\m_i''=(x-i,Y+\sqrt{i})$. Since $k(\m_i)=k(\m_i')=k(\m_i'')=\R$ then it
is clear that $\R[C]$ is $s_c$-normal in $\R[C']$
i.e $$\R[C]=\R[C]^{s_c}_{\R[C']}=\{ f\in\R[C']|\,\,{\rm
  for}\,\,i=1,\ldots,n,\,\,f(\m_i')=f(\m_i'')\}.$$
We set $C_0=C'$ and, for $i=1,\ldots,n$, we set $C_i$ to be the affine plane algebraic
curve over $\R$ with coordinate ring
$\R[C_i]=\R[x,Y_i]/(Y_i^2-x\prod_{j=1}^i(x-j)^2)$. Remark that 
$\R[C]=\R[C_n]\subset \cdots\subset \R[C_1]\subset \R[C_0]=\R[C']$. Since the extension
$\R[C_{i+1}]\to \R[C_i]$ is given by sending $x$ to $x$ and $Y_{i+1}$
to $Y_i (x-(i+1))$ then we have $$\R[C_{i+1}]=\R[C]^{s_c,\m_{i+1}}_{\R[C_i]}=\{ f\in\R[C_i]|\,\,{\rm
  for}\,\,i=1,\ldots,n,\,\,f(\m_{i+1}'\cap \R[C_i])=f(\m_{i+1}''\cap \R[C_i])\}.$$
We have illustrated Theorem \ref{structuralthm} by showing that
$\R[C]$ can be obtained from $\R[C']$ by $n$ successive non-trivial elementary
central gluings. It is clear that the number $n$ of elementary central
gluings is the lowest we can obtain in this case.
\end{ex}

\section{Central seminormalization and localization}

We may wonder if the processes of central seminormalization and localization commute
together. It is known to be true for geometric rings in the special case we take the central seminormalization in the standard integral closure (i.e $B=A'$) and moreover 
we only localize at a central
prime ideal \cite[Thm. 4.23]{FMQ-futur2}. The goal of this section is to show that it is true more generally.

An extension $A\to B$ of rings is called essentially of finite type if $B$ is a localization of $C$ with $A\to C$ an extension of finite type \cite[Def. 53.1]{STPalg}. A domain
$A$ is called Japanese if for any finite extension $\K(A)\to L$ of fields the integral closure of $A$ in $L$ is a finitely generated $A$-module \cite[Def. 159.1]{STPalg}.
A ring $A$ is a Nagata ring if $A$ is Noetherian and for any prime ideal $\p$ of $A$ then the domain $A/\p$ is Japanese \cite[Def. 160.1]{STPalg}.
As a representative example, a finitely generated algebra over a field is a Nagata ring \cite[Prop. 160.3]{STPalg}.

\begin{prop}
  \label{localization1}
Let $A\to B$ be an essentially of finite type extension of domains and assume $A$ is a
Nagata ring. Let $S$ be a
multiplicative closed subset of $A$. If $A$ is $s_c$-normal in $B$
then $S^{-1}A$ is $s_c$-normal in $S^{-1}B$.
\end{prop}

\begin{proof}
Assume $A$ is $s_c$-normal in $B$. By Proposition \ref{intermednorm} $A$ is $s_c$-normal in the integral closure $A_B'$ of $A$ in $B$.
By \cite[Lem. 160.2]{STPalg} $A_B'$ is a finitely generated $A$-module.
By Theorem \ref{structuralthm} there is a finite
sequence $(B_i)_{i=0,\ldots,n}$ of real domains such that:
\begin{enumerate}
\item[1)] $A=B_n\subset \cdots\subset B_1\subset B_0=A_B'$.
\item[2)] $B_1=\widetilde{A}_B'$ is the birational gluing of $A_B'$ over $A$.
\item[3)] For $i\geq 1$, $B_{i+1}$ is the central gluing of $B_i$ over a central
  prime ideal of $A$.
\end{enumerate}
By Proposition \ref{prop7}, $\forall i\geq 1$, $S^{-1}B_{i+1}$ is
$s_c$-normal in $S^{-1}B_i$. By Proposition \ref{transitivity} it
follows that $S^{-1}A$ is $s_c$-normal in 
$S^{-1}B_1=S^{-1}(A_B'\times_{\K(A_B')} \K(A))=(S^{-1}A_B')\times_{\K(A_B')} \K(A)$. Let $S^{-1}A\to
D\to S^{-1}B$ be a sequence of extensions such that $S^{-1}A\to
D$ is $s_c$-subintegral (remark that $D=S^{-1}C$ for $C$ an
intermediate domain between $A$ and $B$). By \cite[Lem. 35.1]{STPalg} then $S^{-1}A_B'$ is the integral closure of $S^{-1}A$ in $S^{-1}B$.
Since a $s_c$-subintegral
extension is birational and integral then $D\subset (S^{-1}A_B')\times_{\K(A_B')} \K(A)$ and thus
$D=S^{-1}A$. It proves that $S^{-1}A$ is $s_c$-normal in $S^{-1}B$.
\end{proof}

Let $A\to B$ be an extension of domains and let $S$ be a
multiplicative closed subset of $A$. Since $A\to A^{s_c,*}_B$ is
$s_c$-subintegral then $S^{-1}A\to S^{-1}(A^{s_c,*}_B)$ is also
$s_c$-subintegral. By Definition \ref{defscnormal}, we get the
following integral extension of domains
\begin{equation}
  \label{equ3}
S^{-1}(A^{s_c,*}_B)\to (S^{-1}A)^{s_c,*}_{S^{-1}B}
\end{equation}
One problem is to know if the extension (\ref{equ3}) is an isomorphism.

\begin{thm} \label{localization2}
Let $A\to B$ be an essentially of finite type extension of domains and assume $A$ is a
Nagata ring. Let $S$ be a
multiplicative closed subset of $A$. Then
$$S^{-1}(A^{s_c,*}_B)= (S^{-1}A)^{s_c,*}_{S^{-1}B}.$$
\end{thm}

\begin{proof}
We already know that $S^{-1}(A^{s_c,*}_B)\subset
(S^{-1}A)^{s_c,*}_{S^{-1}B}$ by (\ref{equ3}). Since $S^{-1}A\to
(S^{-1}A)^{s_c,*}_{S^{-1}B}$ is $s_c$-subintegral then from Lemma
\ref{propcentsubint} it follows that $S^{-1}A\to S^{-1}(A^{s_c,*}_B)$
and $S^{-1}(A^{s_c,*}_B)\to (S^{-1}A)^{s_c,*}_{S^{-1}B}$ are both
$s_c$-subintegral. From Proposition \ref{localization1} then
$S^{-1}(A^{s_c,*}_B)$ is $s_c$-normal in $S^{-1}B$. From above
arguments the proof is done.
\end{proof}

In particular, we get:
\begin{cor} \label{localization2cor}
Let $A\to B$ be an essentially of finite type extension of domains and assume $A$ is a
Nagata ring. Let $\p\in\Sp A$. Then
$$(A^{s_c,*}_B)_{\p}= (A_{\p})^{s_c,*}_{B_{\p}}.$$
\end{cor}

From Corollary \ref{pb1cor} and Theorem \ref{localization2} we
generalize \cite[Thm. 4.23]{FMQ-futur2}.
\begin{cor} \label{corlocalization2}
Let $A\to B$ be a finite and birational extension of domains and assume $A$ is a Nagata ring. Let $S$ be a
multiplicative closed subset of $A$. Then
$$S^{-1}(A^{s_c}_B)= (S^{-1}A)^{s_c}_{S^{-1}B}.$$
\end{cor}

\begin{cor}
  \label{globallocal}
Let $A\to B$ be an essentially of finite type extension of domains and assume $A$ is a
Nagata ring.
We have $$A^{s_c,*}_B=\bigcap_{\p\in\Sp A} (A_{\p})^{s_c,*}_{B_{\p}}.$$
\end{cor}

\begin{proof}
The proof follows from the equality $A^{s_c,*}_B=\bigcap_{\p\in\Sp A}
(A^{s_c,*}_{B})_{\p}$ and Corollary \ref{localization2cor}.
\end{proof}

\section{Central seminormalization of real algebraic varieties}

\subsection{Central seminormalization of affine real algebraic
  variety}
  
In this section, we focus on the existence problem of a central seminormalization of an affine real algebraic variety in another one. Let us introduce the problem. Let $Y\to X$ be a dominant morphism of finite type between two irreducible affine algebraic varieties over $\R$. Does there exists a unique real algebraic variety $Z$ such that $Y\to X$ factorizes through $Z$, satisfying the following property: for any irreducible affine algebraic variety $V$ over $\R$ such that $Y\to X$ factorizes through $V$ then $\pi:V\to X$ is centrally subintegral if and only if $Z\to X$ factorizes through $V$?

\begin{defn} Let $Y\to X$ be a dominant morphism between two affine real algebraic varieties over $\R$. We say that an affine algebraic variety $Z$ over $\R$ is intermediate between $X$ and $Y$ if $Y\to X$ factorizes through $Z$ or equivalently if $\R[Z]$ is intermediate between $\R[X]$ and $\R[Y]$ (by considering the associated ring extension $\R[X]\to\R[Y]$).
\end{defn}

We define the central seminormalization (or $s_c$-normalization) of $X$ in $Y$ as the variety which would give a solution to the problem posed here.
\begin{defn}
Let $Y\to X$ be a dominant morphism of finite type between two irreducible affine real algebraic varieties over $\R$. In case there exists a unique maximal element among the intermediate varieties $V$ between $X$ and $Y$ such that $V\to X$ is centrally subintegral then we denote it by $X_Y^{s_c,*}$ and we call it the central seminormalization or $s_c$-normalization of $X$ in $Y$. In case $Y=X'$ the normalization of $X$ then we omit $Y$ and we call $X^{s_c,*}$ the central seminormalization or $s_c$-normalization of $X$.
\end{defn}
  
We need the following:
\begin{lem}
  \label{lemintermed}
Let $\pi:Y\to X$ be a finite morphism between two irreducible affine
algebraic varieties over $\R$. Let $A$ be a ring such that
$\R[X]\subset A\subset \R[Y]$. Then $A$ is the coordinate ring of a unique irreducible
affine algebraic variety over $\R$ and $\pi$ factorizes through this
variety.
\end{lem}

\begin{proof}
Since $\R[Y]$ is a finite module over the noetherian ring $\R[X]$ then it is a noetherian $\R[X]$-module.
Thus the ring $A$ is a finite $\R[X]$-module as a submodule of a noetherian $\R[X]$-module. It follows that $A$ is a finitely generated algebra over $\R$ and the proof is done.
\end{proof}

Some finiteness results:
\begin{prop} 
\label{lemfiniteness} 
Let $Y\to X$ be a dominant morphism of finite type between two irreducible affine real algebraic varieties over $\R$. The integral closure $\R[X]_{\R[Y]}'$ and the birational and integral closure $\widetilde{\R[X]}_{\R[Y]}'$ of $\R[X]$ in $\R[Y]$ are finite $\R[X]$-modules.
\end{prop}

\begin{proof}
By  \cite[Prop. 160.16]{STPalg} coordinate rings of irreducible affine real algebraic varieties over $\R$ are Nagata domains and thus $\R[X]_{\R[Y]}'$ is a finite $\R[X]$-module by \cite[Prop. 160.2]{STPalg}. The finiteness of $\widetilde{\R[X]}_{\R[Y]}'$ as a $\R[X]$-module is a consequence of Lemma \ref{lemintermed}.
\end{proof}

From the above proposition, we can define the normalization and the birational normalization of a variety in another one.
\begin{defn} \label{defnorm+biratnorminanother}
Let $Y\to X$ be a dominant morphism of finite type between two irreducible affine real algebraic varieties over $\R$. 
\begin{enumerate}
\item[1)] The variety with coordinate ring $\R[X]_{\R[Y]}'$ is called the normalization of 
$X$ in $Y$ and is denoted by $X_Y'$.
\item[2)] The variety with coordinate ring $\widetilde{\R[X]}_{\R[Y]}'$ is called the birational normalization of 
$X$ in $Y$ and is denoted by $\widetilde{X}_Y'$.
\end{enumerate}
\end{defn}

We prove now the existence of a central seminormalization of an affine real algebraic variety in another one.
\begin{thm} \label{thmCSEPvariety}
Let $Y\to X$ be a dominant morphism of finite type between two irreducible affine real algebraic varieties over $\R$. Then, the central seminormalization $X^{s_c,*}_Y$ of $X$ in $Y$ exists and its coordinate ring is the central seminormalization $\R[X]^{s_c,*}_{\R[Y]}$ of $\R[X]$ in $\R[Y]$, namely
$$\R[X^{s_c,*}_Y]=\R[X]^{s_c,*}_{\R[Y]}$$
\end{thm}

\begin{proof}
Assume $X_{reg}(\R)=\emptyset$. By Theorem \ref{pb1}, Propositions \ref{existcentralideal} and \ref{lemfiniteness} then the theorem is proved in that case
and we get $X^{s_c,*}_Y=X_Y'$.

Assume $X_{reg}(\R)\not=\emptyset$ i.e $\R[X]$ is a real domain (Proposition \ref{existcentralideal}). By Theorems \ref{scgeom} and \ref{pb1}, we have to prove that $\R[X]^{s_c,*}_{\R[Y]}$ is a finitely generated algebra over $\R$. By Theorem \ref{pb1}, we have 
$$\R[X]^{s_c,*}_{\R[Y]}=\R[X]^{s_c}_{\widetilde{\R[X]}_{\R[Y]}'}=\R[X]^{s_c}_{\R[\widetilde{X}_Y']}$$
By Lemma \ref{lemfiniteness}, the morphism $\widetilde{X}_Y'\to X$ is finite and thus by Lemma \ref{lemintermed} we get the proof.
\end{proof}

Similarly to the classical normalization, the central normalization is a geometric process associated to an algebraic integral closure. It generalizes \cite[Thm. 4.16]{FMQ-futur2}.
\begin{thm}
\label{thmintclosregulu}
Let $Y\to X$ be a dominant morphism of finite type between two irreducible affine real algebraic varieties over $\R$. 
\begin{enumerate}
\item[1)] If $X_{reg}(\R)\not=\emptyset$ then $\R[X^{s_c,*}_Y]$ is the integral closure of $\R[X]$ in 
$\SRR(\Cent X)\times_{\K(Y)}\R[Y]$.
\item[2)] If $X_{reg}(\R)=\emptyset$ then $\R[X^{s_c,*}_Y]$ is the integral closure of $\R[X]$ in 
$\R[Y]$.
\end{enumerate}
\end{thm}

\begin{proof}
Assume $X_{reg}(\R)=\emptyset$.  Looking at the proof of Theorem \ref{thmCSEPvariety} then we get 2).

Assume $X_{reg}(\R)\not=\emptyset$. From the proof of Theorem \ref{thmCSEPvariety} then we get 
$$\R[X^{s_c,*}_{Y}]=\R[X]^{s_c}_{\R[\widetilde{X}_Y']}$$
From the following commutative diagram
$$\begin{array}{ccccccc}
	\R[X]&\rightarrow &\R[\tilde{X}_Y']&\rightarrow & \R[X_Y']&\rightarrow &\R[Y]\\
	\downarrow&&\downarrow&&&&\\
	\SRR(\Cent X)&\rightarrow &\SRR(\Cent\widetilde{X}_Y') & &\downarrow
	&&\downarrow \\
	\downarrow& &\downarrow&&&&\\
	\K(X)&\simeq&\K(\widetilde{X}_Y')& \rightarrow&\K(X_Y')& \rightarrow &\K(Y)\\
\end{array}$$
we see that the integral closure of $\R[X]$ in $\SRR(\Cent X)\times_{\K(Y)}\R[Y]$ is $\SRR(\Cent X)\times_{\K(X)}\R[\widetilde{X}_Y']$ and we denote this latest domain by $B$. Let $g\in \R[X^{s_c,*}_Y]$.  Let $\pi $ denote the morphism $X^{s_c,*}_Y\to X$.
We have $g\in\R[\widetilde{X}_Y']$. By Theorem \ref{scgeom} there exists $f\in \SRR(\Cent X)$ such that $f\circ \pi=g$ on $\Cent X^{s_c,*}_Y$. It follows that $g\in B$.

By Lemma \ref{lemintermed} then $B=\R[Z]$ for an irreducible affine algebraic variety over $\R$ and we get a finite birational morphism $Z\to X$ factorizing $Y\to X$.
Since $B\subset \SRR(\Cent X)$ then by 5) of Theorem \ref{scgeom} then $\R[X]\to B$ is $s_c$-subintegral and thus $B\subset \R[X^{s_c,*}_Y]$.
\end{proof}

\subsection{Central seminormalization of real schemes}

In this section, we prove the existence of a central seminormalization of a real scheme in another one. It can be seen as a real or central version of Andreotti and Bombieri's construction of the classical seminormalization of a scheme in another one \cite{AB}.
  
\subsubsection{Central locus of a scheme over $\R$}

Let $X=(X,\SO_X)$ be an integral scheme of finite type over $\R$ with field of rational functions $\K(X)$. We say that $x\in X$ is real if the residue field $k(x)$ is a real field. By \cite[Prop. 6.4.2]{Gr1} the residue field at a closed point of $X$ is $\R$ or $\C$, consequently the residue field at a real closed point is the field of real numbers. We denote by $X(\R)$ (resp. $X_{reg}(\R)$) the set of (resp. smooth) real closed points of $X$. We denote by $\eta$ the generic point of $X$, we have $k(\eta)=\K(X)$.  Note that if $U=\Sp A$ is a non-empty affine open subset of $X$ then $U$ is (Zariski) dense in $X$ and $\eta$ is also the generic point of $U$ i.e $A$ is a domain. 
We say that $x\in X$ is central if there exists an affine neighborhood $U=\Sp A$ of $x$ such that $x\in\RCent A$ seeing $x$ as a prime ideal of $A$. We denote
by $\RCent \SO_X$ the set of central points of $X$. By Proposition \ref{existcentralideal} then $\eta$ is central if and only if $\K(X)$ is a real field if and only if $X_{reg}(\R)\not=\emptyset$.
We denote by $\Cent X$ the set of central closed points of $X$, by Theorem \ref{centralNull} and the definition we get $\Cent X= \overline{X_{reg}(\R)}^E$ with $ \overline{X_{reg}(\R)}^E$ denoting the euclidean closure of the set of smooth real closed points.

From Propositions \ref{contractcentral} and \ref{lying-over}, we get:
\begin{prop} \label{lying-overschemes}
Let $\pi:Y\to X$ be a dominant morphism between integral and finitely generated schemes over $\R$. Then $\pi(\RCent \SO_Y)\subset\RCent \SO_X$. If $\pi$ is finite and birational then 
$\pi(\RCent \SO_Y)=\RCent \SO_X$.
\end{prop}

\subsubsection{Normalization and birational normalization of a scheme in another one}
\label{subsecNandB}

Let $\pi:Y\to X$ be a dominant morphism of finite type between integral schemes of finite type over $\R$. The integral closure $(\SO_X)_{\pi_*(\SO_Y)}'$ of $\SO_X$ in $\pi_*(\SO_Y)$ is a coherent sheaf \cite[Lem. 52.15]{STPmorph} and by 
\cite[II Prop. 1.3.1]{Gr2} it is the structural sheaf of a scheme denoted $X_Y'$ called the normalization of $X$ in $Y$. We get a finite morphism $\pi':X_Y'\to X$ factorizing $\pi$.
If $U=\Sp A\subset X$ then $H^0(\pi'^{-1} (U),\SO_{X_Y'})$ is the integral closure of $H^0(U,\SO_X)=A$ in $H^0(\pi^{-1}(U),\SO_Y)$. If $Y=\Sp \K(X)$ then we simply denote $X_Y'$ by $X'$ and we call it the normalization of $X$. We have the universal property that any finite morphism $Z\to X$, with $Z$ an integral scheme over $\R$, factorizing $\pi$ factorizes $\pi'$.

From definition the birational and integral closure $\widetilde{\SO_X}_{\pi_*(\SO_Y)}'$ of $\SO_X$ in $\pi_*(\SO_Y)$ is a quasi-coherent sheaf. Since $\SO_{X_Y'}$ is coherent then it follows from Lemma \ref{lemfiniteness} that $\widetilde{\SO_X}_{\pi_*(\SO_Y)}'$ is also coherent. By \cite[Prop. 1.3.1]{Gr2} it is the structural sheaf of a scheme denoted $\widetilde{X}_Y'$ called the birational normalization of $X$ in $Y$. We have the universal property that any finite and birational morphism $Z\to X$, with $Z$ an integral scheme over $\R$, factorizing $\pi$ factorizes $\widetilde{X}_Y'\to X$.

\subsubsection{Central gluing of a scheme over $\R$ over another one}

Let $X$ be an integral scheme of finite type over $\R$. 
For $x\in X$, we denote by $\m_x$ the maximal ideal of the local ring $\SO_{X,x}$. Since $\K(\SO_{X,x})=\K(X)$ then it follows directly from Definition \ref{defidealcentral} that:
\begin{prop}
Let $x\in X$. Then $x\in\RCent \SO_X$ if and only if $\m_x\in\RCent \SO_{X,x}$.
\end{prop} 

We define an $\SO_X$-algebra that corresponds to the central simultaneous gluings.
\begin{defn} \label{defcentralgluingscheme}
Let $\pi:Y\to X$ be a finite morphism with $Y$ an integral scheme over $\R$. The central gluing of $\pi_*(\SO_Y)$ over $\SO_X$ is the $\SO_X$-subalgebra of $\pi_*(\SO_Y)$, denoted by $(\SO_X)^{s_c}_Y$, whose sections $f\in H^0(U,(\SO_X)^{s_c}_Y)$ on an open subset $U$ of $X$ are the $f\in H^0(U,\pi_*(\SO_Y))$ such that for any $x$ in $U$ then
$$f_x\in \SO_{X,x}+\JRadCe(\pi_*(\SO_Y))_x$$
\end{defn}

Remark that if $x\not\in \RCent \SO_X$ then $\SO_{X,x}+\JRadCe(\pi_*(\SO_Y))_x=(\pi_*(\SO_Y))_x$.

\begin{prop} \label{sheaf}
Let $\pi:Y\to X$ be a finite morphism with $Y$ an integral scheme over $\R$. Then $(\SO_X)^{s_c}_Y$ the central gluing of $\pi_*(\SO_Y)$ over $\SO_X$ is a sheaf.
\end{prop}

\begin{proof}
From Definition \ref{defcentralgluingscheme}, we easily see that $(\SO_X)^{s_c}_Y$ is a presheaf. To get now that it is a sheaf use that $\pi_*(\SO_Y)$ is a sheaf.
\end{proof}

\begin{rem}
Let $\pi:Y\to X$ be a finite morphism with $Y$ an integral scheme over $\R$ and assume $X=\Sp A$ and $Y=\Sp B$ are affine. From above we get 
$$(\SO_X)^{s_c}_Y=A^{s_c}_B$$
\end{rem}

\begin{lem} \label{lemcond}
Let $\pi:Y\to X$ be a finite morphism with $Y$ an integral scheme over $\R$. Let $x\in X$. We have 
$$\m_x(\pi_*\SO_Y)\subset (((\SO_X)^{s_c}_Y)_x:(\pi_*\SO_Y)_x)$$
\end{lem}

\begin{proof}
By \cite[Lem. 2, Ch. 2, Sect. 9]{Ma} and Definition \ref{defcentralgluingscheme}, we get 
$$\m_x(\pi_*\SO_Y)\subset \JRad  (\pi_*\SO_Y)_x\subset \JRadCe (\pi_*\SO_Y)_x\subset ((\SO_X)^{s_c}_Y)_x$$
\end{proof}

\begin{thm} \label{thmcoherent}
Let $\pi:Y\to X$ be a finite morphism with $Y$ an integral scheme over $\R$. Then $(\SO_X)^{s_c}_Y$ is a coherent sheaf on $X$.
\end{thm}

\begin{proof}
It is sufficient to prove $(\SO_X)^{s_c}_Y$ is quasi-coherent since the finiteness property is given by Lemma \ref{lemintermed}. Since the property to be quasi-coherent can be verified locally, we assume $X=\Sp A$, $Y=\Sp B$ with $A$ and $B$ denoting respectively the coordinate rings of $X$ and $Y$. We have now to check the two properties c 1) and  c 2) given by Grothendieck in \cite[I Thm. 1.4.1]{Gr1}.

Let $f\in A$ and set $\D (f)=\{x\in X\mid f\not\in\p_x\}$ (here we identify $x\in X$ with the corresponding prime ideal $\p_x\in\Sp A$). Let $U$ be an open subset of $X$ such that $\D(f)\subset U$ and let $s\in H^0(\D(f),(\SO_X)^{s_c}_Y)$. We have to show that there exists $n\in\N$ such that $(f_{\mid \D(f)})^ns$ extends as a section in $H^0(U,(\SO_X)^{s_c}_Y)$.

For $x\in\D(f)$, we have $s_x\in ((\SO_X)^{s_c}_Y)_x=A_{\p_x}+\JRadCe B_{\p_x}\subset (\pi_*\SO_Y)_x=B_{\p_x}$. Since $\pi_*\SO_Y$ is quasi-coherent then, by \cite[Thm. 1.4.1, d1)]{Gr1} there exists $n\in\N^*$ such that $(f_{\mid \D(f)})^{n-1}s$ extends as a section $t\in H^0(U,\pi_*\SO_Y)$. So, if $x\in\D(f)$ then $t_x\in((\SO_X)^{s_c}_Y)_x=A_{\p_x}+\JRadCe B_{\p_x}$ and if $x\in U\setminus\D(f)$ then $t_x\in (\pi_*\SO_Y)_x=B_{\p_x}$.

We claim that $f_{\mid U}t\in H^0(U,(\SO_X)^{s_c}_Y)$. If $x\in\D(f)$ then clearly $f_xt_x=f_x^n s_x\in  ((\SO_X)^{s_c}_Y)_x$. Assume now $x\in U\setminus\D(f)$. Since $f_x\in \m_x=\p_xA_{\p_x}$ then it follows from Lemma \ref{lemcond} that $f_xt_x\in ((\SO_X)^{s_c}_Y)_x$ and we have proved the claim.

It follows that $(f_{\mid \D(f)})^ns$ extends as a section in $H^0(U,(\SO_X)^{s_c}_Y)$ and we have checked the property c 1) of  \cite[I Thm. 1.4.1]{Gr1}. Since $\pi_*\SO_Y$ is quasi-coherent then obviously $(\SO_X)^{s_c}_Y$ satisfies the property c 2) of  \cite[I Thm. 1.4.1]{Gr1} and the proof is done.
\end{proof}

From \cite[II Prop. 1.3.1]{Gr2} and Theorem \ref{thmcoherent}, we get:
\begin{cor} \label{corthmcoherent1}
Let $\pi:Y\to X$ be a finite morphism with $Y$ an integral scheme over $\R$. There exists an integral scheme $X^{s_c}_Y$ over $\R$, called the central gluing of $Y$ over $X$, with a finite and birational morphism $\pi^{s_c}_Y:X^{s_c}_Y\to Y$ factorizing $\pi$ such that $(\pi^{s_c}_Y)_* \SO_{X^{s_c}_Y}=(\SO_X)^{s_c}_Y$ i.e $$X^{s_c}_Y=\Sp(\SO_X)^{s_c}_Y$$
\end{cor}


\subsubsection{Central seminormalization of a scheme over $\R$ in another one}

From Theorem \ref{pb1}, Propositions \ref{existcentralideal} and above constructions we state the following definition:
\begin{defn}
Let $\pi:Y\to X$ be a dominant morphism of finite type with $Y$ an integral scheme over $\R$. The central seminormalization of $\SO_X$ in $\pi_*(\SO_Y)$ is the $\SO_X$-algebra
denoted by $(\SO_X)^{s_c,*}_Y$ defined by:
\begin{enumerate}
\item[1)] $(\SO_X)^{s_c,*}_Y=(\SO_X)^{s_c}_{\widetilde{X}_Y'}$ if $X_{reg}(\R)\not=\emptyset$.
\item[2)] $(\SO_X)^{s_c,*}_Y=(\SO_X)_{\pi_*(\SO_Y)}'$ if $X_{reg}(\R)=\emptyset$.
\end{enumerate}
\end{defn}

From above constructions and results we are able to prove the existence of the central seminormalization of a real scheme in another one:
\begin{thm} \label{CSEPschemes}
Let $\pi:Y\to X$ be dominant morphism of finite type with $Y$ an integral scheme over $\R$. 
\begin{enumerate}
\item[1)] The $\SO_X$-algebra $(\SO_X)^{s_c,*}_Y$, is a coherent sheaf on $X$.
\item[2)] There exists an integral scheme $X^{s_c,*}_Y$ over $\R$, called the central seminormalization of $X$ in $Y$, with a finite morphism $\pi^{s_c,*}_Y:X^{s_c,*}_Y\to Y$ factorizing $\pi$ such that $(\pi^{s_c,*}_Y)_* \SO_{X^{s_c,*}_Y}=\SO^{s_c,*}_Y$ i.e $$X^{s_c,*}_Y=\Sp\SO^{s_c,*}_{Y}$$
\end{enumerate}
\end{thm}

\begin{proof}
If $X_{reg}(\R)=\emptyset$ then $(\SO_X)^{s_c,*}_Y$, is a coherent sheaf on $X$ by \cite[Lem. 52.15]{STPmorph}. From \ref{subsecNandB} and Theorem \ref{thmcoherent} then
$(\SO_X)^{s_c,*}_Y$, is a coherent sheaf on $X$ in the case $X_{reg}(\R)\not=\emptyset$. The rest of the proof follows from  \cite[II Prop. 1.3.1]{Gr2}.
\end{proof}

\begin{rem}
Let $\pi:Y\to X$ be a dominant morphism of finite type with $Y$ an integral scheme over $\R$ and assume $X=\Sp A$ and $Y=\Sp B$ are affine. From Theorems \ref{pb1} and \ref{CSEPschemes} we get 
$$\SO^{s_c,*}_Y=A^{s_c,*}_B$$
\end{rem}

\begin{defn}
Let $\pi:Y\to X$ be a dominant morphism of finite type with $Y$ an integral scheme over $\R$. We say that $\pi$ is $s_c$-subintegral or centrally subintegral if $\pi$ is finite and the induced map $\RCent \SO_Y\to\RCent \SO_X$ is bijective and equiresidual ($\forall y\in\RCent Y$ we have $k(\pi(y))\simeq k(y)$).
\end{defn}

\begin{rem}
This notion of centrally subintegral morphism has similarities with the concept of universal homeomorphism introduced by Grothendieck \cite[I 3.8]{Gr1}.
\end{rem}

Since a point $x\in X$ is central if and only if it is a central point in an affine neighborhood of $x$ then we derive from Theorems \ref{scgeom} and \ref{thmCSEPvariety}:
\begin{thm}
Let $\pi:Y\to X$ be a dominant morphism of finite type with $Y$ an integral scheme over $\R$. A morphism $Z\to X$ factorizing $\pi$, with $Z$ an integral scheme over $\R$, is centrally subintegral if and only if it factorizes through $\pi^{s_c,*}_Y:X^{s_c,*}_Y\to X$.
\end{thm}



\begin{thebibliography}{ABR2}     




\bibitem{AB} A. Andreotti, E. Bombieri,
{\it Sugli omeomorfismi delle variet\`a algebriche}, 
Ann. Scuola Norm. Sup Pisa (3) 23, 431--450, (1969)

\bibitem{AN} A. Andreotti, F. Norguet, 
{\it La convexit\'e holomorphe dans l'espace analytique des cycles
  d'une vari\'et\'e alg\'ebrique}, 
Ann. Scuola Norm. Sup. Pisa (3) 21, 31--82, (1967)

\bibitem{AM} M. F. Atiyah, I. G. Macdonald, Introduction to
  commutative algebra, Reading: Addison-Wesley, (1969)



\bibitem{BCR} J. Bochnak, M. Coste, M.-F. Roy, 
{\it Real algebraic
  geometry}, Springer, (1998)






\bibitem{Ei} D. Eisenbud, {\it Commutative Algebra with a view toward algebraic geometry}, Graduate texts in mathematics, 150,
Springer-Verlag, 3rd Ed., (2004)

\bibitem{FHMM} G. Fichou, J. Huisman, F. Mangolte,
  J.-P. Monnier,
{\it Fonctions r\'egulues}, J. Reine angew. Math., 718, 103-151 (2016)

\bibitem{FMQ} G. Fichou, J.-P. Monnier, R. Quarez,
{\it Continuous functions on the plane regular after one blowing-up}, 
Math. Z., 285, 287-323, (2017)

\bibitem{FMQ-futur} G. Fichou, J.-P. Monnier, R. Quarez, {\it
    Integral closures in real algebraic geometry}, 
  arXiv:1810.07556, to appear in Journal of Algebraic geometry (2020)

\bibitem{FMQ-futur2} G. Fichou, J.-P. Monnier, R. Quarez, {\it
    Weak and semi normalization in real algebraic geometry},
 arXiv:1706.04467, to appear in Annali della Scuola Normale Superiore di Pisa, Classe di Scienze (2020)



\bibitem{GT} S. Greco, C. Traverso,
{\it On seminormal schemes},
Composition Math. (3) 40, 325-365, (1980)

\bibitem{Gr1} A. Grothendieck, J. A. Dieudonné), {\it \'Eléments de géométrie algébrique I},
Springer, (1971)

\bibitem{Gr2} A. Grothendieck (rédigé avec la collaboration de J. A. Dieudonné), {\it \'Eléments de géométrie algébrique II: \'Etude globale élémentaire de quelques classes de morphismes},
Publ. Math. IHES 8, (1961)

\bibitem{KoKo} J. Koll\'ar, S. Kov\'acs,
{\it Singularities of the minimal model program},
Cambridge Tracts in Mathematics, 200. Cambridge University Press, Cambridge, (2013)

\bibitem{Ko} J. Koll\'ar,
{\it Variants of normality for Noetherian schemes},
Pure Appl. Math. Q. (1) 12, 1-31 (2016)

\bibitem{KuKu2} J. Koll\'ar, W. Kucharz, K. Kurdyka, 
{\it Curve-rational functions},
Math. Ann. 370, 1-2, 39-69 (2018)

\bibitem{KN} J. Koll\'ar, K. Nowak,
{\it Continuous rational functions on real and p-adic varieties},
Math. Z. 279, 1-2, 85-97 (2015).



\bibitem{Ku} W. Kucharz,
{\it Rational maps in real algebraic geometry},
Adv. Geom.~\textbf{9} (4), 517-539, (2009)

\bibitem{KuKu1} W. Kucharz, K. Kurdyka, 
{\it Stratified-algebraic vector bundles},
J. Reine Angew. Math. 745, 105-154 (2018)


\bibitem{LV} J. V. Leahy, M. A. Vitulli, 
{\it Seminormal rings and weakly normal varieties}, 
Nagoya Math. J. (82), 27-56, (1981)





\bibitem{Man} F. Mangolte,
{\it Vari\'et\'es alg\'ebriques r\'eelles},
Cours sp\'ecialis\'es, collection SMF 24, (2017)



\bibitem{Ma} H. Matsumura,
{\it Commutative algebra}, Cambridge studies in advanced mathematics
8, (1989)

\bibitem{Mo} J.-P. Monnier,
{\it Semi-algebraic geometry with rational continuous functions},
Math. Ann. 372, 3-4, 1041-1080 (2018). 









\bibitem{STPalg} {\it The Stack project Algebra}

\bibitem{STPmorph} {\it The Stack project Morphisms}

\bibitem{STPvar} {\it The Stack project Varieties}

\bibitem{Sw} R.G. Swan,
{\it On seminormality},
J. Algebra 67, 210-229, (1980)

\bibitem{T} C. Traverso,
{\it Seminormality and Picard group}, 
Ann. Scuola Norm. Sup. Pisa (3) 24, 585--595, (1970)

\bibitem{Vcor} M. A. Vitulli, 
{\it Corrections to "Seminormal rings and weakly normal varieties"}, Nagoya Math. J.
Vol. 107 (1981), 27-56


\bibitem{V} M. A. Vitulli, 
{\it Weak normality and seminormality}, 
Commutative algebra-Noetherian and non-Noetherian perspectives, 441-480, Springer, New York, (2011)

\end{thebibliography}
\end{document}